\documentclass{article}
\usepackage[utf8]{inputenc}
\usepackage{fullpage}

\usepackage{amsmath,amsthm,amssymb,amsfonts}
\usepackage{mathtools}
\usepackage{tikz}
\usepackage{tikz-cd}
\usepackage{ifthen}
\usepackage[normalem]{ulem}
\usepackage{lipsum}
\usepackage{authblk}

\usepackage{hyperref}

\newtheorem{theorem}{Theorem}
\newtheorem{definition}{Definition}

\newtheorem{remark}{Remark}
\newtheorem{lemma}{Lemma}
\newtheorem{example}{Example}

\newtheorem{corollary}{Corollary}
\theoremstyle{definition}

\title{Computing 1-Periodic Persistent Homology with Finite Windows}

\author{Adam Onus\footnote{\href{a.onus@qmul.ac.uk}{a.onus@qmul.ac.uk}} }
\author{Primoz Skraba\footnote{\href{p.skraba@qmul.ac.uk}{p.skraba@qmul.ac.uk}}}
\affil{School of Mathematical Sciences\\ Queen Mary University of London\\ London, United Kingdom}

\date{}

\usepackage{xcolor}
\usetikzlibrary{shapes}

\newcommand{\D}{\partial}
\newcommand{\Z}{\mathbb{Z}}
\newcommand{\F}{\mathbb{F}}
\newcommand{\R}{\mathbb{R}}

\newcommand{\U}{\mathcal{U}}
\newcommand{\N}{\mathbb{N}}
\newcommand{\tr}{\mathbf{t}}
\newcommand{\vect}{\textbf{vect}}
\newcommand{\aut}{\mathrm{Aut}}
\newcommand{\rank}{\mathrm{rank}}
\newcommand{\im}{\mathrm{im}}

\newcommand{\surj}{\twoheadrightarrow}
\newcommand{\inj}{\hookrightarrow}
\newcommand{\jarrow}{\xrightarrow{j}}
\newcommand{\iarrow}{\xleftarrow{i}}
\newcommand{\gker}{\mathrm{gker}}
\newcommand{\gim}{\mathrm{gim}}

\newcommand{\torus}{\mathbb{T}}
\newcommand{\cS}{\mathbb{S}}
\newcommand{\Hg}{\mathrm{H}}
\newcommand{\Ig}{\mathrm{I}}

\newcommand{\nerve}{\mathcal{N}}
\newtheorem{result}{Result}

\begin{document}

\maketitle

\begin{abstract}
Periodic spaces often arise in practice in simulations and models where boundary effects must be avoided, as well as in naturally periodic settings such as in crystallography. Constructing topological invariants for such spaces which are computable from finite windows is a surprisingly subtle problem -- and one for which we provide insight into in the case of a 1-dimensional periodicity.  
More specifically we study periodic cell complexes, that is a complex $K$ endowed with a covering $q:K\to G$ where $G$ is a finite quotient space of equivalence classes under translations acting on $K$. When $G$ is embedded in a space whose homotopy type is a $d$-torus for some $d$, this  introduces what we call ``toroidal cycles'' in $G$ which do not lift to cycles in $K$ by $q$. Understanding these cycles is critical for understanding the behavior of the periodic space. In this paper, we study both toroidal and non-toroidal cycles for the case $K$ is 1-periodic, i.e. $G=K/\Z$ for some free action of $\Z$ on $K$.
We show that toroidal cycles can be entirely classified by endomorphisms on the homology of unit cells of $K$, and moreover that toroidal cycles have a sense of unimodality when studying the persistent homology of $G$. 
\end{abstract}

\section{Introduction}

Spatially periodic topological spaces arise when simulating or modelling large (possibly infinite) structures on comparably small finite domains without introducing irregular boundary effects.
The first periodic structure one naturally considers arises from the $d$-dimensional integer lattice, i.e. $\mathbb{Z}^d$.
One can then study the resulting structure by embedding it into a flat $d$-dimensional torus, which equivalently represents a single unit cell of translations on $\R^d$ with periodic boundary conditions. See Figure~\ref{fig:intro-fig} for an example where $d=1$.

\begin{figure}[ht]
\centering
\begin{tikzpicture}[x=1cm]

\node (ldots) at (-13.3,1) {$\cdots$};
\node (rdots) at (-3.9,1) {$\cdots$};
\node (project) at (-2,1) {\Huge $\surj$};

\draw (-14+1.24,0.98) to[bend right=60] (-14+1.6,0.93);
\draw (-14+1.24,0.98) to (-14+1.6,0.93);
\draw (-14+1.6,0.93) to[bend left=30] (-14+0.78,0.78);

\foreach \i in {-6,...,-3}
    {
    \draw (2*\i+0.7,1.6) to[bend left=40] (2*\i+1.1,1.36);
    \draw (2*\i+0.99,1.13) to (2*\i+1.1,1.36);
    \draw (2*\i+0.99,1.13) to (2*\i+0.92,0.99);
    \draw (2*\i,1.34) to[bend left=30] (2*\i-0.14,1.08);
    \draw (2*\i+0.71,1.13) to[bend left=20] (2*\i+0.5,1.25);
    \draw (2*\i+0.3,1.2) to[bend left=90] (2*\i+0.15,1);
    \draw (2*\i+0.15,1) to[bend right=90] (2*\i,0.8);
    \draw (2*\i+1.24,0.98) to[bend right=60] (2*\i+1.6,0.93);
    \draw (2*\i+1.24,0.98) to (2*\i+1.6,0.93);
    \draw (2*\i+0.3,1.2) to (2*\i+0.3,2);
    \draw (2*\i+1.6,0.93) to[bend left=30] (2*\i+0.78,0.78);
    \draw (2*\i+1.6,0.93) to (2*\i+1.54,0.67);
    \draw (2*\i+1.7,0.6) to[bend right=30] (2*\i+1.54,0.67);
    
    \draw[fill=gray!30] (2*\i+0.6,0.2) circle (0.15);
    \draw (2*\i+0.3,1.6) circle (0.4);
    \draw[fill=gray!30] (2*\i,0.6) ellipse (0.3 and 0.2);
    \draw[xshift=2*\i cm,rotate=20,fill=gray!30] (1.5,0.7) ellipse (0.05 and 0.2);
    \draw[xshift=2*\i cm,rotate=-30,fill=gray!30] (1,1) ellipse (0.25 and 0.35);
    \draw[xshift=2*\i cm,rotate=-45,fill=gray!30] (-0.3,1.3) rectangle (-0.1,1.5);
    \draw[xshift=2*\i cm,rotate=-45,fill=gray!30] (-0.3,1.3) -- (-0.1,1.3) -- (-0.15,1.2) -- (-0.3,1.3);
    \draw[xshift=2*\i cm,rotate=-45,fill=gray!30] (-0.1,1.3) -- (-0.15,1.2) -- (0,1.1) -- (-0.05,1.35) -- (-0.1,1.3);
    }
\foreach \i in {-7,...,-4}
    {
    \draw[xshift=2*\i cm,rotate=30,fill=gray!30] (2,0) ellipse (0.15 and 0.1);
    \draw (2*\i+1.6,0.93) to[bend right=20] (2*\i+1.7,0.6);
    }

\draw[->,blue,line width=0.25mm] (-10,2.3) -- (-8,2.3);

\draw[blue,dashed,line width=0.2mm] (0,-0.1) -- (0,2.1);
\draw[blue,dashed,line width=0.2mm] (2,-0.1) -- (2,2.1);
\draw[<->,blue,line width=0.35mm] (2,1) arc [x radius = 1.52, y radius = 0.8, start angle = -48.59, end angle = 228.59];

\draw (0.7,1.6) to[bend left=40] (1.1,1.36);
\draw (0.99,1.13) to (1.1,1.36);
\draw (0.99,1.13) to (0.92,0.99);
\draw (2,1.34) to[bend left=30] (1.86,1.08);
\draw (0.71,1.13) to[bend left=20] (0.5,1.25);
\draw (0.3,1.2) to[bend left=90] (0.15,1);
\draw (0.15,1) to[bend right=90] (0,0.8);
\draw (1.24,0.98) to[bend right=60] (1.6,0.93);
\draw (1.24,0.98) to (1.6,0.93);
\draw (0.3,1.2) to (0.3,2);
\draw (1.6,0.93) to[bend right=20] (1.7,0.6);
\draw (1.6,0.93) to[bend left=30] (0.78,0.78);
\draw (1.6,0.93) to (1.54,0.67);
\draw (1.7,0.6) to[bend right=30] (1.54,0.67);

\draw[fill=gray!30] (0.6,0.2) circle (0.15);
\draw (0, 1.34) arc (-138.59:138.59:0.4);
\draw (2, 1.86) arc (138.59:221.41:0.4);
\draw[fill=gray!30] (0,0.4) arc [x radius = 0.3, y radius = 0.2, start angle = -90, end angle = 90];
\draw[fill=gray!30] (2,0.8) arc [x radius = 0.3, y radius = 0.2,
  start angle = 90, end angle = 270];
\draw[rotate=20,fill=gray!30] (1.5,0.7) ellipse (0.05 and 0.2);
\draw[rotate=30,fill=gray!30] (2,0) ellipse (0.15 and 0.1);
\draw[rotate=-30,fill=gray!30] (1,1) ellipse (0.25 and 0.35);
\draw[rotate=-45,fill=gray!30] (-0.3,1.3) rectangle (-0.1,1.5);
\draw[rotate=-45,fill=gray!30] (-0.3,1.3) -- (-0.1,1.3) -- (-0.15,1.2) -- (-0.3,1.3);
\draw[rotate=-45,fill=gray!30] (-0.1,1.3) -- (-0.15,1.2) -- (0,1.1) -- (-0.05,1.35) -- (-0.1,1.3);

\end{tikzpicture}
\caption{Left: a 2-dimensional, 1-periodic cellular complex $K$ in $\R^2$ with translation group $T$ generated by shifts in the direction and magnitude of the thick blue arrow. Right: $K$ is generated by translated copies of the unit cell between the two dashed lines. By gluing the dashed lines together, one obtains the quotient space $G=K/T$, which is a subspace of $\cS^1\times\R$.}
\label{fig:intro-fig}
\end{figure}
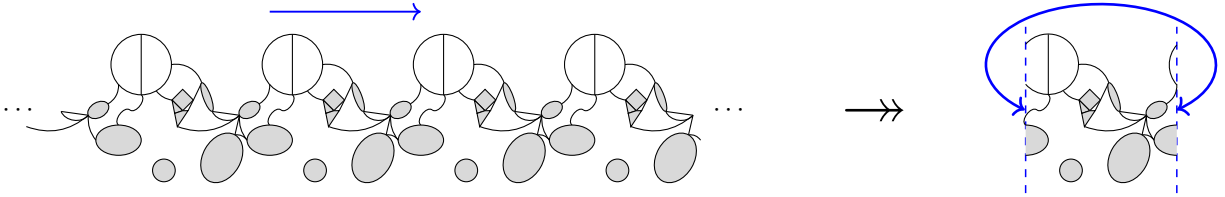

In this paper, we study 1-periodic cellular complexes.
A $d$-periodic cellular complex $K\subseteq \R^n$ is one which is endowed with a translation group, $T\leq \mathrm{Aut}(K)$, of symmetries in $d$ independent directions.
$T$ represents a free action of $\Z^d$ on $K$.
We further define the quotient space $G=K/T$ which is equipped with a natural covering $q:K\to G$.
The group $T$ extends to a group of symmetries on $\R^n$.
Identifying $T\cong\Z^d$ and choosing an appropriate basis for $\R^n$, the corresponding quotient space is $\R^n/T=\torus^d\times\R^{n-d}$ where $\torus^d=(\R/\Z)^d$ is the $d$-dimensional torus.
Thus, for fixed $K$ and $T$ we have the following canonical commutative diagram.

\[
\begin{tikzcd}
         K & {\R^n} & \\
         & & {\torus^d} \\
         G & {\torus^d\times\R^{n-d}} & \\
        \arrow[hook, from=1-1, to=1-2]
        \arrow[hook, dashed, from=3-1, to=3-2]
        \arrow[two heads, "q", from=1-1, to=3-1]
        \arrow[two heads, "q", from=1-2, to=3-2]
        \arrow[two heads, from=3-2, to=2-3]
        \arrow[dashed, from=1-2, to=2-3]
\end{tikzcd}
\]
$K$ may or may not have a filtration, but is often defined over a point cloud $Q$ in application.
$Q$ may represent atoms in a crystal model or molecular dynamics simulation, or more generally any large spatially homogeneous structure.
For example, $Q$ may instead represent pixels or voxels in the digital image analysis of a large porous material with many cavities.

For fixed $K$, neither $T$ nor $q:K\to G$ will be unique, and there is no a priori ``best'' choice which fits all scenarios.
Therefore, we are led to several natural questions: are there invariants of $G$ which do not depend on $T$ for a fixed $K$?
If so, what information about $K$ can we tell from $G$?
What additional information do we need to give $G$ in order to recover all information of $K$ on a computer (i.e. with a finite algorithm)?
If $K$ is filtered, do we have a well-defined notion of persistent homology, and is this theory independent of the choice of $T$ or unit cell?
In this paper, we begin addressing these questions. 

The question of classification has been well-studied in the realm of topological crystallography for the case of periodic graphs \cite{delgado2005three,sunada2012topological,eon2018planar}.
This includes recent work by \cite{edelsbrunner2021density} which has been implemented to find previously unknown degeneracies in the classification of crystal structures.
The question of retrieving the (persistent) homology and topology of $K$ is, conversely, not very well studied and is surprisingly subtle and non-trivial.
Persistent homology is fundamentally not well-defined due to the lack of q-tame-ness of the persistent homology of $K$ \cite{chazal2016structure}, and non-uniqueness of $G$.
As of yet, there has been no successful alternative definition of persistent homology, with most issues stemming from the lack of uniqueness in choosing $T$ and $q$.
Recent work to address this introduced a degree 0 persistent homology theory of periodic point clouds through an extension of merge trees \cite{edelsbrunner2024merge}. Their method employs similar concepts to \cite{onus2022quantifying} to track the asymptotic number of connected components of finite approximations of the point cloud. This approach is independent of the choice of $G$ (respectively $T$), is stable under perturbations and can be computed in polynomial time, however there is no obvious way that this method extends to persistent homology in arbitrary degree.

On the other hand, it is well documented in topological crystallographic literature that different crystal lattices can produce the same quotient graphs \cite{sunada2012topological}.
In \cite{onus2022quantifying}, it is shown that homological features of $K$ may easily be lost in $G$, but also that ``toroidal cycles'' may appear.
Toroidal cycles are phantom homological features of $K$ that appear only in a quotient space $G$.
In other words, toroidal cycles are those cycles representing classes in $\Hg_\bullet(G)$ which are not in the image of the map induced by $q:K\surj G$.
For the case of periodic graphs, in \cite{onus2022quantifying} 
it is shown that one can recover the homology of $K$ by adding appropriate vector weights to the edges of $G$.
For higher dimensions the problem is decidedly harder, with \cite{onus2022quantifying} only providing a heuristic for classifying toroidal cycles relying on spectral sequences.

When the underlying lattice is 1-dimensional, one has a canonical circle-valued map as outlined by the previous commuting diagram, i.e. $K\to\cS^1$ (where $\cS^1=\torus^1$), which is the object of study in Morse-Novikov theory \cite{novikov1991quasiperiodic}.
This is a classical area of algebraic topology where it is known that one can entirely classify the homology of $K$ in this theory via Novikov homology arising from embeddings of the monodromy groups.
These have also been connected to persistent homology, where the indecomposables of the homology groups arising from these maps have been studied \cite{burghelea2013topological,burghelea2017topology}. While clearly related, this theory is not readily adaptable to inferring the homology of all subsets of a periodic space from a small finite subset, i.e. one ``copy" of the space that is repeated.

In a similar context, sheaf-theoretic methods have been used to study indecomposables of cyclic quivers in terms of lifts to a universal cover \cite{fersztand2024harder}.
While our construction in this paper yields a special case of these modules where the cyclic quiver representation may be identified as zig-zag persistence modules, the approach we propose is more computationally efficient (as it is tailored to our setting).
Finally, since $T$ is a free action on $K$, one observes that the cohomology of $G$, $\Hg^\bullet(G)$, is isomorphic to the equivariant cohomology of $K$ \cite{tu2020introductory}. 
Recent work by \cite{adams2024persistent} has also extended this to a persistent equivariant cohomology theory.
Incorporating these techniques will likely be required to generalise our results to spaces with more general group actions, however this theory is otherwise unrelated to our construction.

\paragraph*{Contributions}
Here, we build on the work of \cite{onus2022quantifying} in an algebraic setting in order to study the (persistent) homology of 1-periodic cellular complexes.
In particular, we provide two main results for the classification of toroidal and non-toroidal cycles of finite quotient spaces.
 To do this, we study $K$ (or $G$) locally via maps on the homology of small finite subcomplexes of $K$ induced by inclusion.
We will see that these maps extend to form a zig-zag persistence module, and we can leverage the representation theory to define finite-dimensional endomorphisms which can identify toroidal cycles in  $O(N^\omega\log^2 N)$ time, in which $N$ is the number of cells in a fixed size window, i.e. one copy of the periodic complex, and $\omega$ is the exponent for matrix multiplication. 

Our first result (Theorems~\ref{thm:isomorphisms} and~\ref{thm:recover}) characterizes homology classes which arise due to taking the quotient (but may not exist in all finite windows), i.e. toroidal cycles. We show that for these  cycles,  which appear due to equivalence relations of the ambient torus, 
it is always possible to represent them with a finite window for 1-periodic cell complexes which allows us to find all such cycles efficiently.
\begin{result}
The toroidal cycles of $K$ can be identified with a subspace of the cycles in a unit cell $U$ (respectively, the intersection of adjacent unit cells $V$) of $K$.
Moreover, this subspace is precisely the generalized image of a canonical endomorphism on $\Hg_\bullet(U)$ (on $\Hg_\bullet(V)$).
\end{result}

Our second result (Theorem~\ref{thm:unimodal}) shows there is some partial control over the appearance of toroidal cycles when we take persistence of finite representations. 
\begin{result}
Toroidal cycles are unimodal in the sense that if their local representative at time $i$ represents a non-toroidal cycle at time $j$ for $j>i$ then it must represent a non-toroidal cycle for all $k\geq j$.
\end{result}

\section{Preliminaries}

We first introduce basic notation and definitions required for periodic cellular complexes in an algebraic setting.
By a \textit{cellular complex} we implicitly mean a CW complex \cite{hatcher}, although  the reader may replace this with simplicial or cubical complexes if desired.
For the following definition, we recall that a group $G$ with identity $e$ acts \textit{freely} on a set $X$ if for any $g\in G$ and $x\in X$, $g\cdot x = x$ implies $g=e$. In other words, if a non-trivial group element acts on $X$ then no point of $X$ is fixed.

\begin{definition}
\normalfont
A cellular complex $K$ is $d$-periodic if it is endowed with a free action of $\Z^d$ on $K$ by cellular maps.
\end{definition}
When $d>0$, $K$ will necessarily be infinite, although in this paper we assume in addition that $K$ is locally compact and paracompact.
We denote the group of actions of $\Z^d$ on $K$ by $T\leq \aut(K)$. Further, for $n\geq 1$, let $G_n=K/\sim_n$ denote the quotient space with respect to the equivalence relation $x\sim_n y$ if and only if $\mathbf{t}^n(x)=y$ for some $\mathbf{t}\in T$.
Let $q_n:K\surj G_n$ be the corresponding quotient map.
Note that the free action on a $d$-periodic cellular complex need not be maximal or obey any other additional criteria, so there may be many actions (and many corresponding $T$) which exist for a fixed $K$. We 
will primarily be interested in homology. We will write $\Hg_i(X)$ to mean the $i^\mathrm{th}$ homology group of $X$ and $\Hg_\bullet(X)=\bigoplus_i \Hg_i(X)$.  We always take homology with coefficients in an arbitrary but fixed field $\F$ unless otherwise stated.

\begin{definition} \label{def:toroidal}
\normalfont
A \textit{non-toroidal cycle} in $G_n$ is a cycle representing a class of $\Ig^n:=(q_n)_*(\Hg_\bullet(K))$.
All other cycles in $G_n$ are \textit{toroidal}.
\end{definition}
This definition precisely characterizes the toroidal and non-toroidal cycles -- however, computationally, this definition is not practical as $\Hg_\bullet(K)$ is, in any interesting case, an infinite dimensional object. In Section \ref{sec:toroidal}, we provide an equivalent characterization which does not require an infinite-dimensional object to find the toroidal (and non-toroidal) cycles.

\begin{example} \label{ex:toroidal}
\normalfont
Let $K\subset \R^2$ be the following 1-periodic graph with vertex set $\Z\times\{-1,0,1\}$. We will use this periodic graph as a running example throughout this paper.
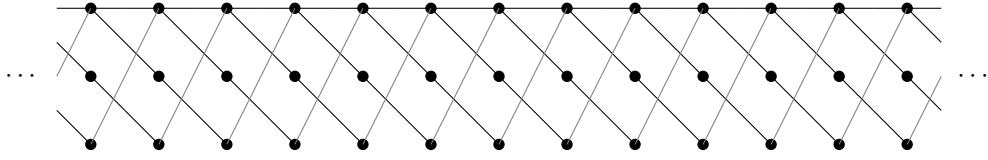
\[
\begin{tikzpicture}[scale=0.9]
\node (ldots) at (-8,2) {$\cdots$};
\node (rdots) at (6,2) {$\cdots$};
\foreach \i in {-7,...,5}
    {
    \foreach \j in {1,...,3}
    	{
    	\node[shape=circle,fill=black,inner sep=0pt,minimum size=0.15cm] (\i,\j) at (\i,\j) { };
    	}
    }
\foreach \i in {-7,...,5}
    {
    \draw (\i-0.5,3) to (\i+0.5,3);
    \draw (\i,3) to (\i+0.5,2.5);
    \draw (\i-0.5,2.5) to (\i,2);
    \draw (\i,2) to (\i+0.5,1.5);
    \draw (\i-0.5,1.5) to (\i,1);
    \draw[gray] (\i,1) to (\i+0.5,2);
    \draw[gray] (\i-0.5,2) to (\i,3);
    }
\end{tikzpicture}
\]
The horizontal integer distance translations $T\cong\Z$ make $K$ 1-periodic.
$K$ naturally covers $G_3$ as shown below.
The green edges indicate an example of a \textit{non-toroidal} 1-cycle of $G_3$ which lifts to a cycle in $K$, whereas the red edges indicate an example of a \textit{toroidal} 1-cycle of $G_3$ which does not lift to a cycle in $K$.
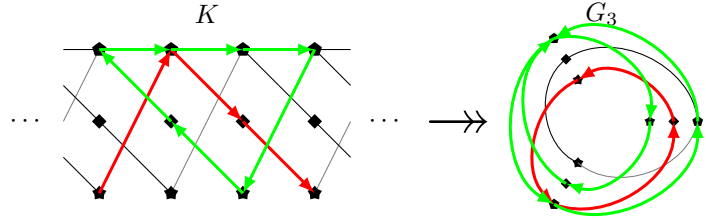
\[
\begin{tikzpicture}[scale=0.95]
\node (K) at (-5.5,1.5) {$K$};
\node (K) at (0,1.5) {$G_3$};
\node (ldots) at (-8,0) {$\cdots$};
\node (rdots) at (-3,0) {$\cdots$};
\node (imply1) at (-2,0) {\Huge $\surj$};
\draw (-7.5,1) to (-3.5,1);
\foreach \i in {-7,...,-4}
    {
    \draw[gray] (\i,-1) to (\i+0.5,0);
    \draw[gray] (\i-0.5,0) to (\i,1);
    \draw (\i,1) to (\i+0.5,0.5);
    \draw (\i-0.5,0.5) to (\i,0);
    \draw (\i,0) to (\i+0.5,-0.5);
    \draw (\i-0.5,-0.5) to (\i,-1);
    \foreach \j/\s in {1/star,2/diamond,3/regular polygon}
    	{
    	\node[shape=\s,fill=black,inner sep=0pt,minimum size=0.2cm] (\i,\j) at (\i,\j-2) { };
    	}
    }
\foreach \i/\j/\k/\l/\m in {-7/1/0/-1/1,-6/0/-1/1/0,-5/-1/1/0/-1}
    {
    \draw[>=latex,red,line width=0.4mm,->] (\i,\l) to (\i+1,\m);
    \draw[>=latex,green,line width=0.4mm,->] (\i,1) to (\i+1,1);
    \draw[>=latex,green,line width=0.4mm,->] (\i+1,\k) to (\i,\j);
    }
\foreach \i in {1,...,3}
    {
    \foreach \j/\s in {1/star,2/diamond,3/regular polygon}
        {
        \node[shape=\s,fill=black,inner sep=0pt,minimum size=0.15cm] (\i:\j) at (120*\i:0.3333+0.3333*\j) { };
        }
    }
    
\draw (120:0.9999) to[bend left = -60] (240:0.6666);
\draw[gray] (240:0.6666) to[bend left = -60] (0:1.3332);
\draw (0:1.3332) to[bend left = -60] (120:0.9999);
\foreach \i/\j/\k/\l/\m in {1/2/1/0/2,2/1/0/2/1,3/0/2/1/0}
    {
    \draw[>=latex,red,line width=0.4mm,->] (120*\i:0.6666+0.3333*\l) to[bend left = -55] (120+120*\i:0.6666+0.3333*\m);
    \draw[>=latex,green,line width=0.4mm,->] (120*\i:1.3332) to[bend left = -60] (120+120*\i:1.3332);
    \draw[>=latex,green,line width=0.4mm,->] (120+120*\i:0.6666+0.3333*\k) to[bend left = 55] (120*\i:0.6666+0.3333*\j);
    }
\end{tikzpicture}
\]
\end{example}

The distinction between toroidal and non-toroidal cycles was introduced in \cite{onus2022quantifying}, and arises since homology classes of $K$ can only be represented by finite linear sums of cells.
This is not a problem in cohomology (unless taking cohomology with compact support) or Borel-Moore homology for example.
For instance, the sum of all (infinitely many) horizontal edges of $K$ in Example~\ref{ex:toroidal} above is \textit{not} an element of $Z_1(K)$, however the sum of $n$ adjacent horizontal edges is an element of $C_1(K)\backslash Z_1(K)$ and projects onto a toroidal 1-cycle in $G_n$.
We are motivated to distinguish between toroidal and non-toroidal cycles, as they are critical in determining finite window size effects and the corresponding implication to $K$.
As another example, a toroidal 2-cycle of a 1-periodic quotient space will lift to a tunnel \textit{through} the infinite space, $K$. 
This type of cavity has different implications to a porous structure than a compact air bubble, for instance, which will always be represented by a non-toroidal cycle.
In order to model spaces with these features, we must be able to distinguish between the two cases.

In Sections~\ref{sec:toroidal} and ~\ref{sec:persistence} we require the notion of a \textit{generalized image} of an endomorphism.
The authors could only find mention of this object in \cite{edelsbrunner2015persistent}, so we give a brief background here of its definition and some basic, yet essential, properties.

\begin{definition} \label{def:gen-imker}
Given a vector space $V$ and an endomorphism $E:V\to V$, the \textit{generalized image} of $E$ is
\[
\gim(E):=\bigcap_{m=1}^\infty\im(E^m),
\]
and similarly, the \textit{generalized kernel} of $E$ is
\[
\gker(E):=\bigcup_{m=1}^\infty\ker(E^m).
\]
\end{definition}

In Theorem~\ref{thm:isomorphisms} and thereafter, we will specifically make use of the following decomposition.
We include the proof of this fact for completeness, as we could not find it in the literature. 
\begin{lemma} \label{lem:gen-imker}
For any finite dimensional vector space $V$ and endomorphism $E:V\to V$, $V$ decomposes as
\[
V = \gker(E)\oplus \gim(E).
\]
\end{lemma}
\begin{proof}
Since $V$ is finite, it is Noetherian, meaning there exists $N\in\N$ for which $\gker(E)=\ker(E^k)$ and $\gim(E)=\im(E^k)$ for every $k\geq N$.
Moreover, by rank-nullity 
\[
\dim(V) = \rank(E^N) + \dim\,\ker(E^N),
\]
where $\rank(E^N)=\dim\,\im(E^N)$.
Thus, we only  need to show $\im(E^N)\cap\ker(E^N)=0$.

Suppose $x\in \im(E^N)\cap\ker(E^N)$,
then $E^N(x)=0$ and there exists $y$ such that $E^N(y)=x$, which combined means that $E^{2N}(y)=0$.
But then $y\in\gker(E)$ which means $y\in\ker(E^N)$, implying $x=0$.
\end{proof}

\begin{corollary} \label{cor:gimker-complexity}
For any finite dimensional vector space $V$ of dimension $d=\dim V$ and endomorphism $E:V\to V$, $\gim(E)$ and $\gker(E)$ are computable in $O(d^\omega\log^2 d)$ time, where $\omega$ is the exponent of matrix multiplication.
\end{corollary}
\begin{proof}
As with the proof of Lemma~\ref{lem:gen-imker} we observe that $\dim(V)=\rank(E^k)+\dim\ker(E^k)$ for every $k\in\N$, and also that there exists a minimal $N$ such that $\gker(E)=\ker(E^k)$ and $\gim(E)=\im(E^k)$ for every $k\geq N$.
Since $\im(E^{k+1})=E^k\left(\im(E)\right)\subseteq \im(E^{k})$, it follows that $\rank(E^{k+1})\leq \rank(E^{k})$ and $\dim\ker(E^k)\leq \dim\ker(E^{k+1})$ for each $k\in\N$. If $\dim\ker(E^k)= \dim\ker(E^{k+1})$ and in turn $\rank(E^{k+1})= \rank(E^{k})$, then $E$ is an automorphism on $\im(E^k)$ and hence $\im(E^k)=\im(E^j)$ for every $j\geq k$. So for $k\leq N$ we must have strict inequalities $\dim\ker(E^k)< \dim\ker(E^{k+1})$ and $\rank(E^{k+1})< \rank(E^{k})$.
In particular, we may bound $N\leq \dim\ker(E^N)\leq d$.
Calculations of images and kernels require matrix multiplication time by first obtaining a triangular decomposition of the matrix (c.f. \cite{bunch1974triangular,milosavljevic2011zigzag}), so the computational complexity will be dominated by at worst calculating $E^{d}$ for $E$ a $d\times d$ matrix.

For $j>0$ an integer, we can construct $E^{2^j}$ by inductively squaring $E^{2^\ell}$ for $\ell\leq j$.
\[
E^{2^{j}} = \left(\cdots\left(\left(\left(E^2\right)^2\right)^2\right)^{\cdots}\right)^2,
\]
which requires exactly $j$ matrix multiplications; a total complexity of $O(jd^\omega)$.
For $d=\sum_{j=0}^{\lfloor \log_2 d\rfloor}\lambda_j2^j$ the binary expansion of $d$ with $\lambda_j\in\{0,1\}$ for each $j$, this means $E^d$ can be computed in
\[
\sum_{j=0}^{\lfloor \log_2 d\rfloor}\lambda_j\,O\left(jd^\omega\right) \leq O\left(d^\omega\log^2 d\right).
\]
\end{proof}

\section{1-periodic complexes} \label{sec:1-period}

We now focus on the case when $K$ is  a 1-periodic cellular complex with respect to translations $T\cong\Z$ generated by some translation $\tr$.
We  introduce an alternative algebraic representation of $K$ from which we can define local endomorphisms which capture the translational symmetries of $K$ (respectively, $G_n$).

Let $U$ to be a subcomplex of $K$ such that $\mathcal{U}=\{\tr^k(U)\,:\,k\in\Z\}$ covers $K$, and let $V:= U\cap \tr(U)$. 
Suppose in addition that $U$ is chosen so that $U\cap \tr^2(U)=\emptyset$, and note that this means $\Hg_\bullet(U)$ and $\Hg_\bullet(V)$ are finite-dimensional since we assume $K$ to be locally compact and paracompact.
Similarly, $\mathcal{U}_n=\{q_n(\tr^k(U))\,:\,k=0,\dots,n-1\}$ will cover $G_n$. Assuming $\Hg_\bullet(U)$ is finite dimensional, 
for $K$ (respectively, $G_n$) this induces an infinite (circular) zigzag persistence module on homologies
\begin{align}
\cdots \iarrow \Hg_\bullet(V) \jarrow \Hg_\bullet(U) \iarrow \Hg_\bullet(V) \jarrow \Hg_\bullet(U) \iarrow \Hg_\bullet(V) \jarrow \cdots
\label{eq:zigzag}
\end{align}
where $i$ is induced by $V\inj U$ and $j$ is induced by the composition of $V\inj \tr(U)$ and the homeomorphism $\tr^{-1}:\tr(U)\to U$.
\begin{example} \label{ex:UV}
\normalfont
Consider $K$ from Example~\ref{ex:toroidal}.
We set $U=K\cap ([0,1]\times [-1,1])$ which in turn makes $V=K\cap (\{1\}\times [-1,1])$.
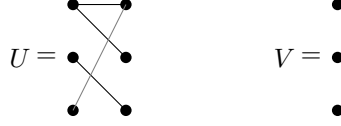
\[
\begin{tikzpicture}[scale = 0.7]
\node (U) at (-1,2) {$U$};
\node (Ueq) at (-0.5,2) {$=$};
\foreach \i in {0,...,1} {\foreach \j in {1,...,3} {\node[shape=circle,fill=black,inner sep=0pt,minimum size=0.15cm] (\i,\j) at (\i,\j) {};}}
\draw (0,3) to (1,3);
\draw (0,3) to (1,2);
\draw (0,2) to (1,1);
\draw[gray] (0,1) to (1,3);
\node (V) at (4,2) {$V$};
\node (Veq) at (4.5,2) {$=$};
\foreach \i in {5,...,5} {\foreach \j in {1,...,3} {\node[shape=circle,fill=black,inner sep=0pt,minimum size=0.15cm] (\i,\j) at (\i,\j) {};}}
\end{tikzpicture}
\]
The natural inclusions $i:V\inj U$ and $j:V\inj \tr(U)$ generate the following sequence of embeddings
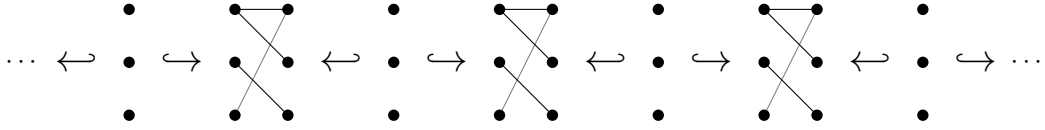
\[
\begin{tikzpicture}[scale = 0.7]
\node (ldots) at (-4,0) {$\cdots$};
\node (rdots) at (15,0) {$\cdots$};
\foreach \i in {-2,3,8,13}
    {
    \node (j,\i-2) at (\i+1,0) {\Large $\inj$};
    \node (i,\i-4) at (\i-1,0) {\Large $\hookleftarrow$};
    \foreach \j in {-1,...,1}
        {
        \node[shape=circle,fill=black,inner sep=0pt,minimum size=0.15cm] (\i,\j) at (\i,\j) {};
        }
    }
\foreach \i in {0,5,10}
    {
    \draw[gray] (\i,-1) to (\i+1,1);
    \draw (\i,1) to (\i+1,1);
    \draw (\i,1) to (\i+1,0);
    \draw (\i,0) to (\i+1,-1);
    \foreach \j in {-1,...,1}
        {
        \node[shape=circle,fill=black,inner sep=0pt,minimum size=0.15cm] (\i,\j) at (\i,\j) {};
        \node[shape=circle,fill=black,inner sep=0pt,minimum size=0.15cm] (\i+1,\j) at (\i+1,\j) {};
        }
    }
\end{tikzpicture}
\]
By taking the 0-dimensional homology in $\F_2$ coefficients (where $\F_2$ is the finite field of 2 elements) with respect to the obvious bases we may write 
\[
\Hg_0(U)\cong \F_2^2\;\;\;\;\;\;\;\Hg_0(V)\cong\F_2^3\;\;\;\;\;\;\;
i \equiv \begin{pmatrix}
1 & 1 & 0 \\
0 & 0 & 1
\end{pmatrix}\;\;\;\;\;\;\; j \equiv \begin{pmatrix}
1 & 0 & 1 \\
0 & 1 & 0
\end{pmatrix}
\]
This induces the following zigzag persistence module from Equation~\ref{eq:zigzag}
\[
\cdots \leftarrow \F_2^3 \xrightarrow{\begin{pmatrix} 1 & 0 & 1 \\ 0 & 1 & 0 \end{pmatrix}} \F_2^2 \xleftarrow{ \begin{pmatrix} 1 & 1 & 0 \\ 0 & 0 & 1 \end{pmatrix}} \F_2^3 \xrightarrow{\begin{pmatrix} 1 & 0 & 1 \\ 0 & 1 & 0 \end{pmatrix}} \F_2^2 \xleftarrow{\begin{pmatrix} 1 & 1 & 0 \\ 0 & 0 & 1 \end{pmatrix}} \F_2^3 \rightarrow \cdots
\label{eq:example}
\]
\end{example}

Suppose we are given $f,g:X\to Y$ in some abelian category.
We may decompose a zigzag diagram akin to (\ref{eq:zigzag}) as follows:
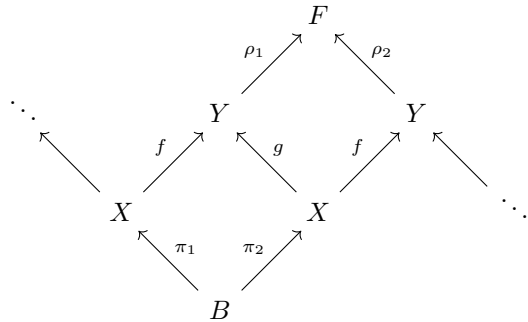
\[
\begin{tikzcd}[column sep={1.3cm,between origins},row sep={1.3cm,between origins}]
         & & & F & &  \\
        \ddots & & Y & & Y &  \\
         & X & & X & & \ddots \\
         & & B & & & 
        \arrow[from=3-2, to=2-1]
        \arrow["f",from=3-2, to=2-3]
        \arrow["g"',from=3-4, to=2-3]
        \arrow["f",from=3-4, to=2-5]
        \arrow[from=3-6, to=2-5]
        \arrow["\pi_1"',from=4-3, to=3-2]
        \arrow["\pi_2",from=4-3, to=3-4]
        \arrow["\rho_1",from=2-3, to=1-4]
        \arrow["\rho_2"',from=2-5, to=1-4]
\end{tikzcd}
\]
Here we define $B$ as the  pullback  of $f$ and $g$. Dually, we define $F$ as the pushout of $f$ and $g$. These are special cases of limits and colimits,  see e.g. \cite[Chapter 2.6]{may1999concise}.
One obtains an explicit representation of these spaces in the category $\vect$ of finite dimensional vector spaces, by setting $B:=\{(x,\tilde{x})\,:\,f(x)=g(\tilde{x})\}\subseteq X\times X$ and $F:=Y\times Y/\Tilde{B}$ where $\tilde{B}=\im\left((f,g):X\to Y\times Y\right)$ such that $[(f(x),g(x)]=0$ in $F$ for every $x\in X$.
Replace $X=\Hg_\bullet(V)$, $Y=\Hg_\bullet(U)$, $f=j$ and $g=i$, and set $B:=B_{i,j}$ and $F:=F_{i,j}$.
The diagram is then:
\[
\begin{tikzcd}[column sep={1.3cm,between origins},row sep={1.3cm,between origins}]
         & & & F_{i,j} & &  \\
        \ddots & & \Hg_\bullet(U) & & \Hg_\bullet(U) &  \\
         & \Hg_\bullet(V) & & \Hg_\bullet(V) & & \ddots \\
         & & B_{i,j} & & & 
        \arrow[from=3-2, to=2-1]
        \arrow["j",from=3-2, to=2-3]
        \arrow["i"',from=3-4, to=2-3]
        \arrow["j",from=3-4, to=2-5]
        \arrow[from=3-6, to=2-5]
        \arrow["\pi_1"',from=4-3, to=3-2]
        \arrow["\pi_2",from=4-3, to=3-4]
        \arrow["\rho_1",from=2-3, to=1-4]
        \arrow["\rho_2"',from=2-5, to=1-4]
\end{tikzcd}
\]
Observe that one can decompose the pullback as
\[
B_{i,j}=\left(\ker(j),0\right)\oplus\left(0,\ker(i)\right)\oplus \mathcal{B},
\]
where $\mathcal{B}\subseteq\Hg_\bullet(V) \times \Hg_\bullet(V)$ is a subspace of $B_{i,j}$ such that for all $(x,y)\in \mathcal{B}$, $x=0$ if and only if $y=0$.
In particular for this representation, $\pi_1$ and $\pi_2$ are projections onto their respective coordinates and $\pi_1|_\mathcal{B},\pi_2|_\mathcal{B}$ are both injective.
Similarly $\Hg_\bullet(V)$ decomposes as $\Hg_\bullet(V)=j^{-1}\left(\im(i)\right)\oplus \mathcal{J}$. 
For $\mathcal{B}$ and $\mathcal{J}$ with this property, we induce an endomorphism on $\Hg_\bullet(V)$.

\begin{lemma} \label{lem:MV}
For fixed choices of $\mathcal{B}$ and $\mathcal{J}$, there exists a unique endomorphism $M_V:\Hg_\bullet(V)\to\Hg_\bullet(V)$ such that for $(x,y)\in\mathcal{B}$ we have $M_V(x)=y$ and $\ker(M_V)=\ker(j)\oplus\mathcal{J}$.
\end{lemma}
\begin{proof}
Observe that $\ker(\pi_1)=(0,\ker(i))$, where
\[
\pi_1(B_{i,j}) = \pi_1\left((\ker(j),0)\oplus(0,\ker(i))\oplus\mathcal{B}\right)=\ker(j)\oplus \pi_1(\mathcal{B}).
\]
Moreover, since $B_{i,j}=\{(x,\tilde{x})\in\Hg_\bullet(V)\times\Hg_\bullet(V)\,:\,j(x)=i(\tilde{x})\}$, we observe when projecting onto the first coordinate that $j^{-1}(\im(i)) = \pi_1(B_{i,j})$.
Thus $\Hg_\bullet(V)=\pi_1(\mathcal{B})\oplus\ker(j)\oplus \mathcal{J}$.
Then for any $x\in\pi_1(\mathcal{B})$ we define $M_V(x)=y$ where $y\in\Hg_\bullet(V)$ is the unique element such that $(x,y)\in\mathcal{B}$ and for $\tilde{x}\in\ker(j)\oplus\mathcal{J}$ we define $M_V(\tilde{x})=0$.
$M_V|_{\pi_1(\mathcal{B})}$ is a well-defined linear map by the vector space structure on $\mathcal{B}$, making it the unique linear map on $\pi_1(\mathcal{B})$ to satisfy the assumptions of the Lemma.
$M_V|_{\ker(j)\oplus\mathcal{J}}=0$ is also well-defined and the unique linear map on $\ker(j)\oplus\mathcal{J}$ to satisfy the assumptions of the Lemma.
Therefore, $M_V$ extends linearly to a unique map on $\Hg_\bullet(V)$ with the desired properties.
\end{proof}

While a fixed choice of $\mathcal{B}$ and $\mathcal{J}$ induces a unique map $M_V$, there is in general no unique $\mathcal{B}$ nor $\mathcal{J}$ with the desired properties.
However, we may canonically choose such subspaces as follows:
$\mathrm{C}_\bullet(U)$ has an inner product structure where for distinct cells $\sigma$ and $\tau$, $\langle \sigma,\tau\rangle = 0$ and $\langle \sigma,\sigma\rangle =1$.
This passes to an inner product structure on both $\Hg_\bullet(U)$ and $\Hg_\bullet(V)$.
Therefore, the notion of orthogonal compliment is well-defined and we may canonically choose $\mathcal{J}=j^{-1}\left(\im(i)\right)^\perp$. Since $\Hg_\bullet(V)=\pi_1(\mathcal{B})\oplus\ker(j)\oplus \mathcal{J}$ and $\pi_1|_{\mathcal{B}}$ is injective, we may also canonically choose $\mathcal{B}$ so that $\pi_1(\mathcal{B})=j^{-1}\left(\im(i)\right)^\perp\cap\ker(j)^\perp$.

Similarly, there is a unique map $\Tilde{M}_V$ which one obtains through the same procedure, only substituting $i$ and $j$.
That is, for $\mathcal{I}=i^{-1}\left(\im(j)\right)^\perp$ and $\mathcal{D}$ such that $\pi_2(\mathcal{D})\subseteq \ker(i)^\perp$, $\ker(\Tilde{M}_V)=\ker(i)\oplus\mathcal{I}$ and $\Tilde{M}_V(y)=x$ for every $(x,y)\in\mathcal{D}$.

Now consider the pullback $\tilde{B}_{i,j}$ of $\Hg_\bullet(U)\xrightarrow{\rho_1} F_{i,j}\xleftarrow{\rho_2} \Hg_\bullet(U)$.
That is,
\[
\begin{tikzcd}[column sep={1.4cm,between origins},row sep={1.4cm,between origins}]
         & F_{i,j} & \\
        \Hg_\bullet(U) & & \Hg_\bullet(U) \\
         & \tilde{B}_{i,j} & \\
         & \Hg_\bullet(V) & \\
        \arrow["\rho_1",from=2-1, to=1-2]
        \arrow["\rho_2"',from=2-3, to=1-2]
        \arrow["\tilde{\pi}_1"',from=3-2, to=2-1]
        \arrow["\tilde{\pi}_2",from=3-2, to=2-3]
        \arrow["i"',from=4-2, to=2-1]
        \arrow["j",from=4-2, to=2-3]
        \arrow["\alpha",from=4-2, to=3-2]
\end{tikzcd}
\]
where $\alpha$ is the unique map satisfying the commutative diagram.
By construction, $\tilde{B}_{i,j}=\{(y,z)\in\Hg_\bullet(U)\times\Hg_\bullet(U)\,:\,\rho_1(y)=\rho_2(z)\}$.
However, $\rho_1(y)=\rho_2(z)$ if and only if $[(y,0)]=[(0,z)]$ in $F_{i,j}$ if and only if there exists $x\in\Hg_\bullet(V)$ such that $y=i(x)$ and $z=j(x)$.
Thus, we identify $\alpha=(i,j):\Hg_\bullet(V)\to \Tilde{B}_{i,j}\subseteq \Hg_\bullet(U)\times \Hg_\bullet(U)$, which is surjective.

As with $B_{i,j}$ we have a decomposition
\[
\tilde{B}_{i,j}=\left(\ker(\rho_1),0\right)\oplus\left(0,\ker(\rho_2)\right)\oplus\tilde{\mathcal{B}},
\]
where $\tilde{\mathcal{B}}$ is a subspace of $\tilde{B}_{i,j}$ such for all $(y,z)\in\tilde{\mathcal{B}}$, $y=0$ if and only if $z=0$.
Observe $\ker(\rho_1)=i(\ker(j))$ and $\ker(\rho_2)=j(\ker(i))$.
Then analogous to the case for $V$, we may construct an endomorphism $M_U$ based on the additional choice of decomposition of $\Hg_\bullet(U)=\im(i)\oplus \tilde{\mathcal{J}}$, where $\im(\tilde{\pi}_1)=\im(i)$.

\begin{lemma}
For fixed choices of $\tilde{\mathcal{B}}$ and $\Tilde{\mathcal{J}}$, there exists a unique endomorphism $M_U:\Hg_\bullet(U)\to\Hg_\bullet(U)$ such that for $(y,z)\in\tilde{\mathcal{B}}$ we have $M_U(y)=z$ and $\ker(M_V)=i(\ker(j))\oplus\tilde{\mathcal{J}}$.
\end{lemma}
\begin{proof}
This is an identical construction to Lemma~\ref{lem:MV} with the decomposition
\[
\Hg_\bullet(U) = \Tilde{\pi}_1(\tilde{\mathcal{B}})\oplus i(\ker(j))\oplus\Tilde{\mathcal{J}}.
\]
\end{proof}

Like with $V$, we may define $\tilde{\mathcal{B}}$ and $\Tilde{\mathcal{J}}$ in terms of the orthogonal compliment.
Explicitly in this case, we will have $\Tilde{\mathcal{J}}=\im(i)^\perp$ and $\Tilde{\pi}_1(\Tilde{\mathcal{B}}) = i\left(\ker(j)\right)^\perp\cap \im(i)^\perp$.
There is, again, also a unique map $\Tilde{M}_U$ obtained through the same procedure, only substituting $i$ and $j$.

\begin{example} \label{ex:braid}
\normalfont Consider $K$ from Example~\ref{ex:toroidal} and set $U$ and $V$ as in Example~\ref{ex:UV}.
Continuing from Equation~\ref{eq:example} and choosing subspaces $\mathcal{B},\tilde{\mathcal{B}},\mathcal{J},\Tilde{\mathcal{J}}$ induced by the orthogonal compliment gives
\[
M_U\equiv \begin{pmatrix}
1 & 1 \\
0 & 0
\end{pmatrix}\;\;\;\;\;\;\;\;\;\; M_V \equiv \begin{pmatrix}
1 & 0 & 1 \\
0 & 0 & 0 \\
0 & 1 & 0
\end{pmatrix}.
\]
In particular, note that 
\[
M_U\circ j = \begin{pmatrix}
1 & 1 & 1 \\
0 & 0 & 0
\end{pmatrix} = j \circ M_V,
\]
and also that
\[
M_U\begin{pmatrix}
1 \\ 0
\end{pmatrix} = \begin{pmatrix}
1 \\ 0
\end{pmatrix}\;\;\;\;\; M_U\begin{pmatrix}
1 \\ 1
\end{pmatrix} = 0\;\;\;\;\; M_V\begin{pmatrix}
1 \\ 0 \\0
\end{pmatrix}=\begin{pmatrix}
1 \\ 0 \\0
\end{pmatrix}\;\;\;\;\; M_V\begin{pmatrix}
1 \\ 0 \\1
\end{pmatrix}=0\;\;\;\;\; M_V^2\begin{pmatrix}
1 \\ 1 \\ 0
\end{pmatrix}=0
\]
meaning that $\gim(M_V)$, $\gim(M_U)$ and $\gker(M_V)$ are one dimensional and $\gker(M_U)$ is two-dimensional.
Moreover, the basis elements for each space are related via $i$ and $j$ as follows:
\[
i\begin{pmatrix}
1 \\ 0 \\0
\end{pmatrix} = j\begin{pmatrix}
1 \\ 0 \\0
\end{pmatrix} = \begin{pmatrix}
1 \\ 0
\end{pmatrix} \;\;\;\;\;\;\;\; i\begin{pmatrix}
1 \\ 0 \\ 1
\end{pmatrix} = j\begin{pmatrix}
1 \\ 1 \\ 0
\end{pmatrix} = \begin{pmatrix}
1 \\ 1
\end{pmatrix} \;\;\;\;\;\;\;\; i\begin{pmatrix}
1 \\ 1 \\0
\end{pmatrix} = j\begin{pmatrix}
1 \\ 0 \\1
\end{pmatrix} = 0.
\]

\end{example}
We point out that $M_U$ and $M_V$ are related since both are defined in terms of the same two maps induced by inclusion.
These relations will allow us to leverage both algebraic and topological symmetries, however there is some subtlety involved which we characterise in the following Lemmas.

\begin{lemma} \label{lem:alg-maps}
Fix $\mathcal{B},\tilde{\mathcal{B}},\mathcal{J},\Tilde{\mathcal{J}}$ as above.
If $j\left(\mathcal{J}\right)\subseteq \Tilde{\mathcal{J}}$ and $j\left(\mathcal{B}\right)=\{(j(x),j(y))\,:\,(x,y)\in\mathcal{B}\}\subseteq (i(\ker(j)), 0)\oplus \tilde{\mathcal{B}}$
\[
M_U \circ j = j \circ M_V.
\]
\end{lemma}

\begin{proof}
Recall $\Hg_\bullet(V)=\pi_1(\mathcal{B})\oplus\ker(j)\oplus \mathcal{J}$ and $\ker(\rho_1)=i(\ker(j))$.
Thus we need only show that $M_U \circ j = j \circ M_V$ holds on $\ker(j)$, $\mathcal{J}$ and $\pi_1(\mathcal{B})$.

(1) If $x\in\ker(j)$ then $j(x)=0$ and $M_V(x)=0$, so $M_U\circ j(x) = j\circ M_V(x)$ trivially.

(2) If $x\in\mathcal{J}$ then $M_V(x)=0$ and $j(x)\in\Tilde{\mathcal{J}}$ by assumption, so $M_U\circ j(x)=0=j\circ M_V(x)$.

(3) If $x\in \pi_1(\mathcal{B})$ then there exists unique $y\in\Hg_\bullet(V)$ such that $(x,y)\in\mathcal{B}$ and $j\circ M_V(x):=j(y)$.
On the other hand, $(x,y)\in\mathcal{B}$ means $j(x)=i(y)$ and $j(\mathcal{B})\subseteq (i(\ker(j)), 0)\oplus \tilde{\mathcal{B}}$ by assumption.
Thus $M_U\circ j(x)=M_U\circ i(y)$ and 
\[
j(x,y)= (j(x),j(y)) =(i(y),j(y))=(i(y_0),0) + (i(y_1),j(y_1)),
\]
where $y_0\in\ker(j)$ and $(i(y_1),j(y_1))\in\tilde{\mathcal{B}}$.
In particular, $M_U\circ i(y)=M_U\circ i(y_0)+M_U\circ i(y_1)=j(y_1)$ by construction, and $j(y)=j(y_0)+j(y_1)=j(y_1)$.
So $j\circ M_V(x)=j(y) = M_U\circ j(x)$.
\end{proof}

Under similar constraints, we will also have $\Tilde{M}_U \circ i = i \circ \Tilde{M}_V$.
We also note that the construction of $\mathcal{B},\tilde{\mathcal{B}},\mathcal{J},\Tilde{\mathcal{J}}$ via orthogonal compliments satisfies the hypothesis for Lemma~\ref{lem:alg-maps}.
Therefore, for the remainder of this paper we assume this construction.

\begin{remark} \label{rem:zig-zag} Within persistence theory, a natural question is how the choice of basis above relates to the resulting basis given by the algorithms used to compute zigzag persistence \cite{carlsson2010zigzag,milosavljevic2011zigzag} and the resulting representatives \cite{dey2021updating,dey2025fast}. This construction directly also relates to that of \cite{burghelea2013topological} for circle-valued maps. Here we broadly discuss the relationship, but stress that this is purely an aside and is not used further.  
Recall the standard structure theorem of zigzag persistence modules to Equation~\ref{eq:zigzag} yields a decomposition in terms of interval modules. 
That is, the zigzag persistence module can be decomposed into irreducible persistence modules of the form $I[a,b]$ for $-\infty\leq a\leq b\leq \infty$, where $I[a,b](c)=\mathbb{F}$ for $c\in[a,b]$ and $0$ otherwise, and all maps are identity on $I[a,b]$ and $0$ otherwise.
Numbering the resulting zigzag by the natural numbers, we set even integer coordinates to $\Hg_\bullet(V)$  and the odd integer coordinates to $\Hg_\bullet(U)$.
Let $I[i,j](k)$ denote the intervals $[i,j]$ evaluated at the zigzag at $k$ and $I[i,j](k\rightarrow \ell)$ denote the same evaluated on the image $k\rightarrow \ell$, so a subspace of the space at $\ell$. Finally, we denote the lift of $k\rightarrow \ell$ by $(k\rightarrow \ell)^{-1}$, so a subspace at $k$. This leads to another coherent choice of $\mathcal{B},\tilde{\mathcal{B}},\mathcal{J},\Tilde{\mathcal{J}},$ given by the conditions below:
\begin{itemize}
    \item $I[a,0](0)\subseteq \ker(j)$ for all $a\leq 0$,
    \item $I[a,1](0)\subseteq\mathcal{J}$ for all $a\leq 0$,
    \item $I[a,b](0)\subseteq\pi_1(\mathcal{B})$ for all $a\leq 0$ and $b\geq 2$, where $M_V|_{I[a,b](0)}=I[a,b](1\to 2)^{-1}\circ I[a,b](0\to 1) $, 
    \item $I[a,1](1)\subseteq\tilde{\mathcal{J}}$ for all $a\leq 1$,
    \item $I[a,2](1)\subseteq i(\ker(j))$ for all $a\leq 1$,
    \item $I[a,b](1)\subseteq\tilde{\pi}_1(\tilde{\mathcal{B}})$ for all $a\leq 1$ and $b\geq 3$, where $M_U|_{I[a,b](1)}=I[a,b](2\to 3)\circ I[a,b](1\to 2)^{-1} $.
\end{itemize}
In particular, this decomposition will agree with the orthogonal compliment construction.
\end{remark}

Lemma~\ref{lem:alg-maps} is an algebraic property of the construction of $M_U$ and $M_V$, whereas Lemma~\ref{lem:geom-maps} below describes a geometric property inherited from the interpretation of the maps $i$ and $j$ being induced by inclusion.
In order to state this, and for the rest of this paper, we write $\alpha\sim\gamma$ to indicate that $\alpha$ and $\gamma$ are homologous unless otherwise stated.

\begin{lemma} \label{lem:geom-maps}
\normalfont
Suppose $[\alpha],[\alpha']\in \Hg_\bullet(U)$ and $[\theta],[\theta']\in \Hg_\bullet(V)$.
Then for all $k\in\N$ the following hold:
\begin{enumerate}
\item If $M_U^k[\alpha]=[\alpha']$ and $M_U^\ell[\alpha]\in\im(i)$ for each $\ell=0,\dots,k-1$ then $\alpha\sim\tr^{k}(\alpha')$.
\item If $\Tilde{M}_U^k[\alpha]=[\alpha']$ and $\Tilde{M}_U^\ell[\alpha]\in\im(j)$ for each $\ell=0,\dots,k-1$ then $\alpha\sim\tr^{-k}(\alpha')$.
\item If $M_V^k[\theta]=[\theta']$ and $ M_V^\ell[\theta]\in j^{-1}\left(\im(i)\right)$ for each $\ell=0,\dots,k-1$ then $\theta\sim\tr^k(\theta')$.
\item If $\Tilde{M}_V^k[\theta]=[\theta']$ and $\tilde{M}_V^{\ell}[\theta]\in i^{-1}\left(\im(j)\right)$ for each $\ell=0,\dots,k-1$ then $\theta\sim\tr^{-k}(\theta')$.
\end{enumerate}
\end{lemma}

\begin{proof}
We have the following cases for $M_U$ and $M_V$ for $k=1$ below.
Applying this inductively we obtain the results for $k>1$. 
The proofs for $\tilde{M}_U$ and $\tilde{M}_V$ are identical.
\begin{enumerate}
\item If $M_V[\theta]=[\theta']$ and $[\theta]\in j^{-1}\left(\im(i)\right)$ then there exists $\theta_0$ and $\theta_1$ such that $[\theta]=[\theta_0]+[\theta_1]$ with $[\theta_0]\in\ker(j)$ and $([\theta_1],[\theta'])\in\mathcal{B}$.
This means $j[\theta]=j[\theta_1]=i[\theta']$ in $\Hg_\bullet(U)$.
But $i$ is induced by $U\hookleftarrow V$ and $j$ is induced by $V\inj\tr(U)$, so $[\tr^{-1}(\theta)] = [j(\theta)] = [i(\theta')] = [\theta']$ in $\Hg_\bullet(U)$, and in turn $[\theta]=[\tr(\theta')]$. Similarly, the natural embedding $U\inj K$ then implies $[\theta]=[\tr(\theta')]$ in $\Hg_\bullet(K)$, or equivalently $\theta\sim\tr(\theta')$.
\item If $M_U[\alpha]=[\alpha']$ and $[\alpha]\in \im(i)$ then there exists cycles $\gamma_0,\gamma_1$ of $V$ such that $[\alpha]=i[\gamma_0]+i[\gamma_1]$ with $[\gamma_0]\in \ker(j)$, and $(i[\gamma_1],j[\gamma_1])\in\Tilde{\mathcal{B}}$ such that $j[\gamma_1]=[\alpha']$.
In other words, $[\alpha]=i[\gamma_0+\gamma_1]$ and $j[\gamma_0+\gamma_1]=[\alpha_1]$.
Interpreting $i,j$ as being induced by inclusions as with the first proof, it follows that $[\alpha]=[\tr(\alpha')]$ in $\Hg_\bullet(K)$, or equivalently $\alpha\sim\tr(\alpha')$.
\end{enumerate}
\end{proof}

\begin{remark}
\normalfont
The converse of all the statements of Lemma~\ref{lem:geom-maps} are in general false.
For instance, in Example~\ref{ex:UV}, all connected components of $U$ and $V$ are homogeneous in $K$.
Instead, we interpret $M_U$ as projecting $\Hg_\bullet(U)$ onto a local representative of the same cycle in $\Hg_\bullet(\tr(U))$, with similar interpretations for all other endomorphisms.
\end{remark}

\section{Toroidal cycles} \label{sec:toroidal}
Recall from Definition~\ref{def:toroidal} that ``toroidal cycles'' are cycles of $G_n$ which are \textit{not} in the image of the homology map induced by $q_n:K\to G_n$. We apply the machinery developed above to characterizing toroidal cycles. 
By localizing toroidal cycles, we will be able to classify them entirely by the generalized image of $M_V$ or $M_U$.
\begin{theorem} \label{thm:isomorphisms}
There exists $\varphi_1,\varphi_2,\varphi_3,\varphi_4$ such that
\[
\begin{tikzcd}
{\gim(M_U)} & {\gim(M_V)} & {\Hg_\bullet(G_n)/\Ig^n} & {\gim(\tilde{M}_V)} & {\gim(\tilde{M}_U)}
\arrow["\varphi_3"', from=1-2, to=1-1]
\arrow["\varphi_4", from=1-4, to=1-5]
\arrow[hook',"\varphi_1"', from=1-3, to=1-2]
\arrow[hook,"\varphi_2", from=1-3, to=1-4]
\end{tikzcd}
\]
where $\varphi_3$ and $\varphi_4$ are isomorphisms.
\end{theorem}
The proof is constructive, as we explicitly define the maps in the proof. Before proving Theorem~\ref{thm:isomorphisms}, however, we first need to introduce some important facts about the structure of toroidal cycles.
First, for any $[\alpha]\in \Hg_\bullet(G_n)$ there is a (not necessarily 
unique) $\tilde{\alpha}\in C_\bullet(K)$ for which
\begin{itemize}
\item $q_n(\tilde{\alpha})=\alpha$.
\item There is a set $S$ and for each $k\in S$ there is $c_k\in \F$, $m_k\in\Z$ and a cell $\tilde{\sigma_k}$ such that \[ \D\tilde{\alpha} = \sum_{k\in S} c_k\left(\tilde{\sigma_k}-\tr^{n\cdot m_k}(\tilde{\sigma_k})\right). \]
\item The cardinality of $S$ is finite and minimal with respect to these properties.
\end{itemize}
Non-uniqueness follows from the fact that if $\tilde{\alpha}$ satisfies these properties then so do $\tr^n(\tilde{\alpha})$ and $\tr^{-n}(\tilde{\alpha})$.
Conversely, $\alpha$ is non-toroidal if and only if $S$ is empty. Indeed, if $S=\emptyset$ then $\tilde{\alpha}$ is a cycle, so $[\alpha]\in\Ig^n$, and if $S\neq\emptyset$ then minimality of $S$ means there are no cycles in $K$ which map to $\alpha$, so $[\alpha]\not\in\Ig^n$.

These facts are essential for Theorem~\ref{thm:isomorphisms}, where we will leverage the above representative of $\D\tilde{\alpha}$ to represent toroidal cycles faithfully in $\gim(M_V)$.
Lemma~\ref{lem:geom-maps} above essentially shows that $M_U$ and $M_V$ keep track of the action of the fundamental group of $\mathbb{S}^1$ on $K$, also known as the holonomy \cite{zein2010local}.
At a high level, we will show that if $[\gamma]$ is in the generalized image of $M_U$ or $M_V$ then the orbit of the homology class $[\tilde{\gamma}]$ by the lift of the fundamental group action is finite (although we will not use this language in doing so). 
In particular, we will be able to identify $\tilde{\gamma}$ and its orbit with $\D\tilde{\alpha}$ for some toroidal cycle $\alpha$.

We now have all the required machinery to prove Theorem~\ref{thm:isomorphisms}.
\begin{proof}[Proof of Theorem~\ref{thm:isomorphisms}]
By symmetry, we need only consider the maps $\varphi_1$ and $\varphi_3$.

\noindent \textbf{($\mathbf{\varphi_1}$):} 
Consider the Mayer-Vietoris spectral sequence (see Appendix~\ref{sec:MVSS}) obtained by covering $G_n$ with $\mathcal{U}_n$ mentioned previously.
Recall that we also assume $U\cap \tr^2(U)=\emptyset$, meaning at most pairwise intersections of the cover are non-trivial and we need only consider the $p=0$ and $p=1$ columns of the spectral sequence.
We have an isomorphism $\phi_n$ between $\Hg_\bullet(G_n)$ and the homology of the total complex $\Hg_\bullet(T_{\mathcal{U}_n,G_n})$, where each $[\alpha]\in \Hg_\bullet(G_n)$ corresponds to 
\[
\Hg_\bullet(T_{\mathcal{U}_n,G_n})\ni[(\gamma_0,\gamma_1)]=\left[\left(\sum_{\ell=0}^{n-1}\alpha_\ell,\sum_{\ell=0}^{n-1}\theta_\ell\right)\right],
\]
where $\alpha_\ell\in C_\bullet(\tr^\ell(U))$, $\theta_\ell\in C_\bullet(\tr^\ell(V))$, $\D^0\gamma_0=\D^1\gamma_1$ and $\alpha\sim\sum_{\ell=0}^{n-1}\alpha_\ell$ in $G_n$. This is precisely the statement of Theorem~\ref{thm:mvss} in the Appendix -- specifically see Equation~\ref{eq:correspondence}.
Under this identification, if $\alpha\sim\alpha'$ in $G_n$ then $[(\gamma_0,\gamma_1)]=[(\gamma'_0,\gamma'_1)]$ in $\Hg_\bullet(T_{\mathcal{U}_n,G_n})$, which in turn implies $\theta_\ell\sim\theta'_\ell$ in $\tr^\ell(V)$ for each $\ell=0,\dots,n-1$.
Thus $\psi_\ell:\Hg_\bullet(G_n)\to \Hg_\bullet(V)$, where $\psi_\ell[\alpha]=[\theta_\ell]$, is a well-defined homomorphism for each $\ell$.
Since $\Hg_\bullet(V)$ is finite dimensional, $\Hg_\bullet(V)=\gker(M_V)\oplus\gim(M_V)$.
We will show that $\psi_0$ maps non-toroidal cycles, $\Ig^n$, to $\gker(M_V)$ and so passes to an injective homomorphism $\tilde{\psi}_0:\Hg_\bullet(G_n)/\Ig^n\inj \gim(M_V)$.

First, if $[\alpha]\in \Ig^n$ then lifting to $K$, we have the following commutative diagram
\[
\begin{tikzcd}
	{\Ig^n} & {\Hg_\bullet(T_{\mathcal{U}_n,G_n})} & {\Hg_\bullet(V)} \\
	{\Hg_\bullet(K)} & {\Hg_\bullet(T_{\mathcal{U},K})} & {\bigoplus_{\ell\in\Z}\Hg_\bullet(V)}
	\arrow[hook', from=1-1, to=1-2]
	\arrow["\cong", from=2-1, to=2-2]
	\arrow[from=1-2, to=1-3]
	\arrow[from=2-2, to=2-3]
	\arrow["\psi_0", bend left = 15, from=1-1, to=1-3]
	\arrow[two heads, from=2-1, to=1-1]
	\arrow[two heads, from=2-3, to=1-3]
\end{tikzcd}
\]
which maps elements as follows
\[
\begin{tikzcd}
	{[\alpha]} & {\left[\left(\sum_{\ell=0}^{n-1}\alpha_\ell,\sum_{\ell=0}^{n-1}\theta_\ell\right)\right]} & {[\theta_0]} \\
	{[\tilde{\alpha}]} & {\left[\left(\sum_{\ell=M}^{N}\tilde{\alpha}_\ell,\sum_{\ell=M}^{N}\tilde{\theta}_\ell\right)\right]} & {\sum_{M/n < \ell < N/n}[\tilde{\theta}_{n\ell}]}
	\arrow[from=1-1, to=1-2]
	\arrow[from=2-1, to=2-2]
	\arrow[from=1-2, to=1-3]
	\arrow[from=2-2, to=2-3]
	\arrow[from=2-1, to=1-1]
	\arrow[from=2-3, to=1-3].
\end{tikzcd}
\]
In the diagram above, for a fixed choice of cycle $\tilde{\alpha}$ in $K$ which maps to $\alpha$ in $G_n$, we use $M$ and $N$ to denote the minimal and maximal integers respectively so that $\tilde{\alpha}$ is supported on $\bigcup_{\ell=M}^N\tr^\ell(U)$.
By construction of the total complex, note that $j(q_n(\tilde{\theta}_\ell)) = q_n(\tilde{\alpha_\ell})=i(q_n(\tilde{\theta}_{\ell+1}))$, which by construction of $M_V$ in terms of orthogonal compliments means $M_V[q_n(\tilde{\theta}_\ell)]=[q_n(\tilde{\theta}_{\ell+1})]$ for each $\ell$.
Specifically, $M_V[q_n(\tilde{\theta}_N)]=0$, so by an inductive argument on each $\ell$ it follows that $[\theta_0]\in\gker(M_V)$ and hence $\psi_0[\alpha]\in\gker(M_V)$.

We have therefore established that $\tilde{\psi}_0:\Hg_\bullet(G_n)/\Ig^n\inj \gim(M_V)\cong \Hg_\bullet(V)/\gker(M_V)$, where $\tilde{\psi}_0[\alpha]=[\theta_0] + \gker(M_V)$, is well-defined.
To show that it is injective, we need only show that if $[\alpha]\not\in \Ig^n$ then $\psi_0[\alpha]=[\theta_0]\not\in \gker(M_V)$.
We first consider the generalization of the above commuting diagram:
\[
\begin{tikzcd}
	{\Hg_\bullet(G_n)} & {\Hg_\bullet(T_{\mathcal{U}_n,G_n})} & {\Hg_\bullet(V)} \\
	{q_n^{-1}(Z(G_n))} & {E^0_{0\bullet}\bigoplus Z(E^0_{1\bullet})} & {\bigoplus_{\ell\in\Z}\Hg_\bullet(V)}
	\arrow["\cong", from=1-1, to=1-2]
	\arrow[hook, from=2-1, to=2-2]
	\arrow[from=1-2, to=1-3]
	\arrow[two heads, from=2-2, to=2-3]
	\arrow["\psi_0", bend left = 20, from=1-1, to=1-3]
	\arrow[two heads, from=2-1, to=1-1]
	\arrow[two heads, from=2-3, to=1-3]
\end{tikzcd}
\]
which now maps elements as follows
\[
\begin{tikzcd}
	{[\alpha]} & {\left[\left(\sum_{\ell=0}^{n-1}\alpha_\ell,\sum_{\ell=0}^{n-1}\theta_\ell\right)\right]} & {[\theta_0]} \\
	{\tilde{\alpha}} & {\left(\sum_{\ell=M}^{N}\tilde{\alpha}_\ell,\sum_{\ell=M}^{N}\tilde{\theta}_\ell\right)} & {\sum_{M/n < \ell < N/n}[\tilde{\theta}_{n\ell}]}
	\arrow[from=1-1, to=1-2]
	\arrow[from=2-1, to=2-2]
	\arrow[from=1-2, to=1-3]
	\arrow[from=2-2, to=2-3]
	\arrow[from=2-1, to=1-1]
	\arrow[from=2-3, to=1-3].
\end{tikzcd}
\]
In the diagram above, similar to before, we use $M$ and $N$ to denote the minimal and maximal integers respectively so that $\tilde{\alpha}$ is supported on $\bigcup_{\ell=M}^N\tr^\ell(U)$.
Let $S$ be a finite set described earlier such that $\D \tilde{\alpha}=\sum_{k\in S}c_k(\tilde{\sigma_k}-\tr^{n\cdot m_k}(\tilde{\sigma_k}))$.
$S$ will be non-empty, so $g=\mathrm{gcd}(n\cdot m_k)$ and $l=\mathrm{lcm}(\cdot m_k)$ are well-defined.
By construction of the total complexes, we have $\tilde{\alpha}_\ell=j(\tilde{\theta}_\ell)=i(\tilde{\theta}_{\ell+1})$ for each $\ell=M,M+1,\dots,N$.
The addition of periodic copies of $\tilde\alpha$ will be ``glued'' to the boundary of $\tilde{\alpha}$, so we can extend this relation so that $j(\tilde{\theta}_N)=\tr^l(\tilde{\alpha}_\ell)=\tr^l(i(\tilde{\theta})_\ell))$ for some $\ell$.
Thus $M_V^l[q_n(\tilde{\theta}_\ell)]=[q_n(\tilde{\theta}_\ell)]$.
By this relation we will have $M_V^l[\theta_0]=[\theta_0]$, so $[\theta_0]\in \gim(M_V)$.

\vspace{0.2cm}

\noindent \textbf{($\mathbf{\varphi_3}$):} Recall from Lemma~\ref{lem:alg-maps} that $M_U\circ j = j\circ M_V$.
Since $\Hg_\bullet(U)$ and $\Hg_\bullet(V)$ are finite dimensional, there must exist some $N\in\N$ for which $\gim(M_V)=\im(M_V^N)$ and $\gim(M_U)=\im(M_U^N)$.
Moreover, by Lemma~\ref{lem:gen-imker}, $M_V$ will be an isomorphism on $\im(M_V^N)$ and likewise for $M_U$.

If $x\in\ker(j)\cap\im(M_V^N)$, then by construction $M_V^N(x)=M_V(x) = 0$, implying $x=0$, so $j$ is injective on $\im(M_V^N)$.
By repeated application of Lemma~\ref{lem:alg-maps}, $M_U^N\circ j = j\circ M_V^N$, and we observe that
\[
j\left(\im(M_V^N)\right) = (M_U^N)(\im(j)) \subset \im(M_U^N).
\]
Hence, $j$ specifically embeds $\im(M_V^N)$ into $\im(M_U^N)$.
Further, $\im(M_U^N)=\im(M_U^{N+1})=\gim(M_U)$, and since $M_U$ is an isomorphism on $\im(M_U^N)$, for every $x\in \im(M_U^N)$ there exists a $y\in\im(M_U^N)$ such that $M_U^{N+1}y=x$.
However, by construction, $M_U(y) = j(z)$ for some $z\in \Hg_\bullet(V)$.
Together this implies that $x = M_U^N\circ j(z) = j\circ M_V^N(z)$.
Hence, $\varphi_3=j|_{\im(M_V^N)}$.

\end{proof}

We note that the maps $M_V$ and $M_U$ encode the translational symmetry on the generalized image. 
The fact that $\varphi_3$ is a restriction of $j$ also implies that $\gim(M_U)\subseteq \im(i)$ and $j(\gim(M_V))\subseteq\im(i)$.

\begin{corollary} \label{cor:general-image}
If $[\alpha],[\alpha']\in\gim(M_{U})$, $[\beta],[\beta']\in\gim(\tilde{M}_{U})$, $[\theta],[\theta']\in\gim(M_{V})$ and $[\delta],[\delta']\in\gim(\tilde{M}_{V})$ then
\begin{enumerate}
    \item If $M_{U}^n[\alpha]=[\alpha']$ then $\alpha\sim \tr^n(\alpha')$.
    \item If $\tilde{M}_{U}^n[\beta]=[\beta']$ then $\beta\sim \tr^{-n}(\beta')$.
    \item If $M_{V}^n[\theta]=[\theta']$ then $\theta\sim \tr^n(\theta')$.
    \item If $\tilde{M}_{V}^n[\delta]=[\delta']$ then $\delta\sim \tr^{-n}(\delta')$.
\end{enumerate}
\end{corollary}
\begin{proof}
This follows immediately from Lemma~\ref{lem:geom-maps} and the fact $\gim(M_U)\subseteq \im(i)$ and $j(\gim(M_V))\subseteq\im(i)$.
\end{proof}
As an immediate consequence, we conclude the following result which relates $M_U$ and $M_V$ with $\tilde{M}_U$ and $\tilde{M}_V$ respectively.
\begin{corollary} \label{cor:equivalences} The following relations hold
\begin{align*}
\gim(M_V) & = \gim(\tilde{M}_V) := Z_V\\
\gim(M_U) & = \gim(\tilde{M}_U) := Z_U \\
(M_V|_{Z_V})^{-1} & = \tilde{M}_V|_{Z_V}\\
(M_U|_{Z_U})^{-1} & = \tilde{M}_U|_{Z_U}
\end{align*}
\end{corollary}

More directly, the fact that $\varphi_1$ in Theorem~\ref{thm:isomorphisms} is an embedding tells us that all toroidal cycles in $G_n$ can be identified with an element of $\gim(M_V)$ (or $\gim(M_U)$).
In fact, the construction of $\varphi_1:\Hg_\bullet(G_n)/\Ig^n\inj \gim(M_V)$ can easily be used to recover all toroidal cycles with the aid of Corollary~\ref{cor:general-image}, as illustrated below.
\begin{theorem} \label{thm:recover}
All toroidal cycles of $G_n$ can be recovered from $\gim(M_V)$ independent of $n$.
\end{theorem}
\begin{proof}
Suppose $\alpha$ is a toroidal cycle of $G_n$ and consider the map $\varphi_1$ constructed in Theorem~\ref{thm:isomorphisms}.
Then $\varphi_1[\alpha]:=[\theta]\neq0$ and $M_V^n[\theta]=[\theta]$ by the construction of $\varphi_1$.
Conversely, if $M_V^n[\theta]=[\theta]$ and $[\theta]\neq 0$ then $[\theta]\in\gim(M_V)$, so $\tilde\theta\sim\tr^n(\tilde\theta)$ by Corollary~\ref{cor:general-image}, where we consider $\tilde\theta$ as a cycle in $K$.
Choosing $\tilde{\alpha}_\theta$ so that $\D\tilde{\alpha}_\theta=\tr^n(\tilde{\theta})-\tilde{\theta}$ means that $\alpha_\theta:=q_n(\tilde{\alpha}_\theta)$ will be a toroidal cycle in $G_n$ -- where we observe that $\alpha_\theta$ is a cycle and $\varphi_1[\alpha_\theta]=[\theta]\in\gim(M_V)$ by construction.
In particular, if we set $[\theta_\ell]:=M_V^\ell[\theta]$ for $\ell\geq0$, then there exists $\gamma_\ell$ so that $\D\gamma_\ell=\tr(\theta_{\ell+1})-\tr(\theta_\ell)$, and we may write
\[
\alpha_\theta=\sum_{\ell=0}^{n-1}\tr^\ell(\gamma_\ell)
\]
modulo the addition of non-toroidal cycles of $G_n$.
Therefore, all toroidal cycles in $G_n$ can be constructed in this way and will correspond to non-zero elements of the $1$-eigenspace of $M_V^n$, which is a subspace of $\gim(M_V)$ for each $n$.
\end{proof}
If $N_U$ is the number of cells in $U$ then $\Hg_\bullet(U)$, $\Hg_\bullet(V)$, $i,j$ and $M_V$ can all be calculated in $O((N_U)^\omega)$ time.
Therefore, one can calculate $\gim(M_V)$ and hence entirely classify toroidal cycles in $O((N_U)^\omega\log^2 N_U)$ time.

\section{Persistence} \label{sec:persistence}
Suppose that $K$ is now a filtered 1-periodic cellular complex.
This means that a cell $\sigma$ is born if and only if $\mathbf{t}(\sigma)$ is born simultaneously for every $\mathbf{t}\in T$.
Equivalently, given $K\surj G_1$, every filtration on $G_1$ admits a filtration on $K$ where a cell is born if and only if the corresponding equivalence class in $G_1$ is born.
We then have the following commutative diagram
\[
\begin{tikzcd}
	& \vdots & \vdots & \vdots & \vdots & \\
	\cdots & {\Hg_\bullet(U_{m+2})} & {\Hg_\bullet(V_{m+2})} & {\Hg_\bullet(U_{m+2})} & {\Hg_\bullet(V_{m+2})} & \cdots \\
	\cdots & {\Hg_\bullet(U_{m+1})} & {\Hg_\bullet(V_{m+1})} & {\Hg_\bullet(U_{m+1})} & {\Hg_\bullet(V_{m+1})} & \cdots \\
	\cdots & {\Hg_\bullet(U_m)} & {\Hg_\bullet(V_m)} & {\Hg_\bullet(U_m)} & {\Hg_\bullet(V_m)} & \cdots \\
	\cdots & {\Hg_\bullet(U_{m-1})} & {\Hg_\bullet(V_{m-1})} & {\Hg_\bullet(U_{m-1})} & {\Hg_\bullet(V_{m-1})} & \cdots \\
	& \vdots & \vdots & \vdots & \vdots & 
	\arrow[from=6-2, to=5-2]
	\arrow[from=5-2, to=4-2]
	\arrow[from=4-2, to=3-2]
	\arrow[from=3-2, to=2-2]
	\arrow[from=3-3, to=2-3]
	\arrow[from=4-3, to=3-3]
	\arrow[from=5-3, to=4-3]
	\arrow[from=6-3, to=5-3]
	\arrow[from=6-4, to=5-4]
	\arrow[from=5-4, to=4-4]
	\arrow[from=4-4, to=3-4]
	\arrow[from=3-4, to=2-4]
	\arrow[from=6-5, to=5-5]
	\arrow[from=5-5, to=4-5]
	\arrow[from=4-5, to=3-5]
	\arrow[from=3-5, to=2-5]
	\arrow[from=4-5, to=4-6]
	\arrow[from=5-5, to=5-6]
	\arrow[from=5-5, to=5-4]
	\arrow[from=5-3, to=5-4]
	\arrow[from=5-3, to=5-2]
	\arrow[from=5-1, to=5-2]
	\arrow[from=4-1, to=4-2]
	\arrow[from=4-3, to=4-2]
	\arrow[from=4-3, to=4-4]
	\arrow[from=4-5, to=4-4]
	\arrow[from=3-5, to=3-4]
	\arrow[from=3-5, to=3-6]
	\arrow[from=3-3, to=3-4]
	\arrow[from=3-3, to=3-2]
	\arrow[from=3-1, to=3-2]
	\arrow[from=2-1, to=2-2]
	\arrow[from=2-2, to=1-2]
	\arrow[from=2-3, to=1-3]
	\arrow[from=2-3, to=2-2]
	\arrow[from=2-3, to=2-4]
	\arrow[from=2-5, to=2-4]
	\arrow[from=2-5, to=2-6]
	\arrow[from=2-5, to=1-5]
	\arrow[from=2-4, to=1-4]
\end{tikzcd}
\]
Along each row we construct the same maps, $M_{V_m}$ and $M_{U_m}$, however these will not necessarily commute with the vertical maps.
This is due to the choice of $\mathcal{J}_m$ for the decomposition $\Hg_\bullet(V_m)=j_m^{-1}(\im(i_m))\oplus \mathcal{J}_m$, not being functorial.
That is, there is in general no consistent basis such that the following maps commute and in particular commutativity breaks at the arrow struck out below.
\[
\begin{tikzcd}
        \Hg_\bullet(V_m) & \Hg_\bullet(V_{m+1}) \\
        \mathcal{J}_m & \mathcal{J}_{m+1}  \\
	\arrow[from=1-1, to=1-2]
	\arrow[red, "/"{anchor=center,sloped}, from=2-1, to=2-2]
	\arrow[hook', from=2-1, to=1-1]
	\arrow[hook, from=2-2, to=1-2]
\end{tikzcd}
\]
In the example below, we see a simple case where the functoriality of $M_{V_m}$ and $M_{U_m}$ fails.

\begin{example} \label{ex:persistence}
\normalfont
Let $K$ be as in Example~\ref{ex:toroidal}.
We define a filtration on $U$ as follows
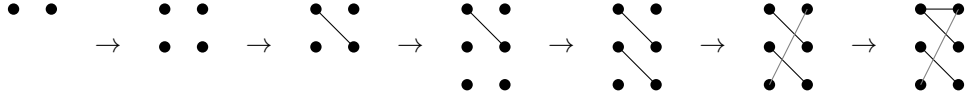
\[
\begin{tikzpicture}
\foreach \i in {0,...,6}
    {
    \node[shape=circle,fill=black,inner sep=0pt,minimum size=0.15cm] (\i+1,1,1) at (2*\i,1) { };
    \node[shape=circle,fill=black,inner sep=0pt,minimum size=0.15cm] (\i+1,1,2) at (0.5 + 2*\i,1) { };
    \ifthenelse{\i<6}{\node (\i+1,arrow) at (1.25 + 2*\i,0.5) {$\rightarrow$};}{ }

    \ifthenelse{\i>0}{\node[shape=circle,fill=black,inner sep=0pt,minimum size=0.15cm] (\i+1,2,1) at (2*\i,0.5) { };}{ }
    \ifthenelse{\i>0}{\node[shape=circle,fill=black,inner sep=0pt,minimum size=0.15cm] (\i+1,2,2) at (0.5+2*\i,0.5) { };}{ }
     
    \ifthenelse{\i>2}{\node[shape=circle,fill=black,inner sep=0pt,minimum size=0.15cm] (\i+1,3,1) at (2*\i,0) { };}{ }
    \ifthenelse{\i>2}{\node[shape=circle,fill=black,inner sep=0pt,minimum size=0.15cm] (\i+1,3,2) at (0.5+2*\i,0) { };}{ }
            
    \ifthenelse{\i>1}{\draw (2*\i,1) to (0.5+2*\i,0.5);}{ }
            
    \ifthenelse{\i>3}{\draw (2*\i,0.5) to (0.5+2*\i,0);}{ }
            
    \ifthenelse{\i>4}{\draw[gray] (2*\i,0) to (0.5+2*\i,1);}{ }
            
    \ifthenelse{\i>5}{\draw (2*\i,1) to (0.5+2*\i,1);}{ }
    }
\end{tikzpicture}
\]
which induces the following filtration on $V$
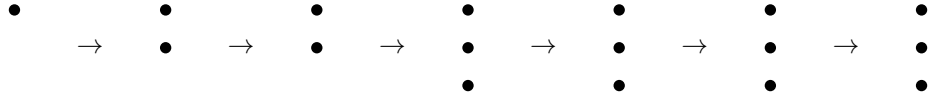
\[
\begin{tikzpicture}
\foreach \i in {0,...,6}
    {
    	\node[shape=circle,fill=black,inner sep=0pt,minimum size=0.15cm] (\i,1) at (2*\i,1) {};
     
    	\ifthenelse{\i>0}{\node[shape=circle,fill=black,inner sep=0pt,minimum size=0.15cm] (\i,2) at (2*\i,0.5) {};}{}
     
    	\ifthenelse{\i>2}{\node[shape=circle,fill=black,inner sep=0pt,minimum size=0.15cm] (\i,3) at (2*\i,0) {};}{}
     
    	\ifthenelse{\i<6}{\node (\i+1,arrow) at (1 + 2*\i,0.5) {$\rightarrow$};}{}
    }
\end{tikzpicture}
\]
Giving $\Hg_\bullet(V)$ the obvious basis, with $\F_2$ coefficients, the persistence maps $\iota_{ij}$ are given by
\[
\iota_{ij} \equiv \begin{cases}\hfil
\begin{pmatrix}
1 \\ 0
\end{pmatrix} \hfil& i=1,\; j=2,3 \\ \hfil
\begin{pmatrix}
1 \\ 0 \\ 0
\end{pmatrix} \hfil& i=1,\; j>3 \\
\begin{pmatrix}
1 & 0 \\ 0 & 1 \\ 0 & 0
\end{pmatrix} & i=2,3,\; j>3 \\\hfil
\mathrm{id}\hfil & \text{otherwise}
\end{cases}.
\]
We also have $M_{V_1}\equiv M_{V_2}\equiv 0$ and 
\begin{align*}
M_{V_3} \equiv \begin{pmatrix}
0 & 0 \\
1 & 0
\end{pmatrix}&\;\;\;\;
M_{V_4} \equiv \begin{pmatrix}
0 & 0 & 0 \\
1 & 0 & 0 \\
0 & 0 & 0
\end{pmatrix}\;\;\;\;
M_{V_5} \equiv \begin{pmatrix}
0 & 0 & 0 \\
1 & 0 & 0 \\
0 & 1 & 0
\end{pmatrix}\\
M_{V_6} \equiv &\begin{pmatrix}
0 & 0 & 1 \\
1 & 0 & 0 \\
0 & 1 & 0
\end{pmatrix}\;\;\;\;
M_{V_7} \equiv \begin{pmatrix}
1 & 0 & 1 \\
0 & 0 & 0 \\
0 & 1 & 0
\end{pmatrix}.
\end{align*}
One easily notices that $\iota_{23}\circ M_{V_2} = 0$, but $M_{V_3}\circ \iota_{23}=M_{V_3}\neq 0$, for instance.
More generally, $M_{V_i}$ and $M_{V_j}$ will only commute with $\iota_{i,j}$ if either $i=j$, $(i,j)=(1,2)$ or $(i,j)=(3,4)$.
\end{example}

While $\iota_{ij}\circ M_{V_i}\neq M_{V_j}\circ \iota_{ij}$ in general, we saw in Lemma~\ref{lem:geom-maps} and Corollary~\ref{cor:general-image} that the $M_{V_i}$ at least respect translational symmetries, which are also respected by persistence.
Thus we at least still expect these endomorphisms to respect the persistent homology of $\Hg_\bullet(G_n)$ in some way.
In particular, once we introduce a toroidal cycle in the filtration we expect at least one toroidal cycle to persist throughout the filtration since $G_n$ is embedded in a space homotopic to $\mathbb{S}^1$.
In this sense, a toroidal cycle can only die if it combines with another toroidal cycle, in which case we often gain non-toroidal cycles.
This intuition indeed holds true and can be made more precise as a form of unimodality.

\begin{theorem} \label{thm:unimodal}
Suppose $\{G_k\}$ is a filtration of the quotient space $G$ which induces filtrations $\{U_k\}$ and $\{V_k\}$ on $U$ and $V$ respectively.
Then $\{M_{U_k}\}$ and $\{M_{V_k}\}$ are unimodal in the sense that for all $k\leq \ell\leq m$,
\begin{align*}
\gker(M_{U_\ell})\cap\gim(M_{U_k}) \;\;&\subseteq\;\;\gker(M_{U_m})\cap\gim(M_{U_k}) \\
\gker(M_{V_\ell})\cap\gim(M_{V_k}) \;\;&\subseteq\;\;\gker(M_{V_m})\cap\gim(M_{V_k}).
\end{align*}
\end{theorem}

\begin{proof}
We prove the second inclusion, and then a similar argument also proves the first.
We proceed as follows.
First, we apply Corollary~\ref{cor:general-image} to show that elements of $\gker(M_{V_\ell})\cap\gim(M_{V_k})$ are homologous to their images under $M_{V_t}$ for any parameter $t\geq k$.
We then show that if $[\theta]\in \gker(M_{V_\ell})\cap\gim(M_{V_k})$, then the difference $M_{V_m}[\theta]-M_{V_\ell}[\theta]$ lies in $\gker(M_{V_m})$, which we use to complete the proof.

Suppose $[\theta]\in\gker(M_{V_\ell})\cap\gim(M_{V_k})$.
By Corollary~\ref{cor:general-image}, for every $l$ this means $\theta\sim\tr^l(\theta_l)$ in $G_n$, where $M_{V_k}^l[\theta]:=[\theta_l]$, and in particular $j_k[\theta_l]=i_k[\theta_{l+1}]$ in $\Hg_\bullet(U_k)$.
Now for $s\leq t$, set $\iota_{st}:\Hg_\bullet(G_s)\to\Hg_\bullet(G_t)$ to be the  map  induced by the inclusion $G_s\inj G_t$ in the filtration on $G$.
For $t\geq k$ we have $i_t\circ\iota_{kt}=\iota_{kt}\circ i_k$ and $j_t\circ\iota_{kt}=\iota_{kt}\circ j_k$, so $j_t[\theta_l]=i_t[\theta_{l+1}]$ for each $l$ as well.
Now recall that $\Hg_\bullet(V_t)=j_t^{-1}(\im(i_t))\oplus\mathcal{J}_t$, which leads to the construction of $M_{V_t}$ in Lemma~\ref{lem:MV} where $\ker(M_{V_t})=\mathcal{J}_t$, and $j_t(x)=i_t\circ M_{V_t}(x)$ for each $x\in j_t^{-1}(\im(i_t))$.
This means $M_{V_t}[\theta_l]-[\theta_{l+1}]\in\ker(i_t)$ for each $l$ and each $t\geq k$.
Note that the requirement that $j_k[\theta_l]=i_k[\theta_{l+1}]$ is key here, as otherwise these differences may lie in $\mathcal{J}_t$.
This means $\Lambda_l:=M_{V_\ell}[\theta_l]-[\theta_{l+1}]\in\ker(i_\ell)$ and $\Delta_l:= M_{V_m}[\theta_l]-[\theta_{l+1}]\in\ker(i_m)$.
Then observe
\begin{align*}
M_{V_m}[\theta_l] & = [\theta_{l+1}] + \Delta_l\\
M_{V_\ell}[\theta_l] & = [\theta_{l+1}] + \Lambda_l\\
M_{V_m}^{r+1}[\theta_l] & = M_{V_m}^r\circ M_{V_m}[\theta_l] = M_{V_m}^r[\theta_{l+1}] + M_{V_m}^r(\Delta_l) \\
M_{V_\ell}^{r+1}[\theta_l] & = M_{V_\ell}^r\circ M_{V_\ell}[\theta_l] = M_{V_\ell}^r[\theta_{l+1}] + M_{V_\ell}^r(\Lambda_l)
\end{align*}
and so by induction
\begin{align*}
M_{V_m}^N[\theta] & = [\theta_{N}] + \sum_{l=0}^{N-1}M_{V_m}^{N-l-1}\left(\Delta_l\right) \\
M_{V_\ell}^N[\theta] & = [\theta_{N}] + \sum_{l=0}^{N-1}M_{V_m}^{N-l-1}\left(\Lambda_l\right)
\end{align*}
which combine to give 
\[
M_{V_m}^N[\theta]=M_{V_\ell}^N[\theta]+\sum_{l=0}^{N-1}M_{V_m}^{N-l-1}\left(\Delta_l-\Lambda_l\right).
\]

Now we recall that $[\theta]\in\gker(M_{V_\ell})$ by assumption, so $M_{V_\ell}^N[\theta]=0$ for sufficiently large $N$.
To show that $[\theta]\in\gker(M_{V_m})$ (and hence complete the proof), it suffices to show that $\Delta_l-\Lambda_l\in\gker(M_{V_m})$ for each $l$.
Indeed, $x\in\gker(M_{V_m})$ if and only if $M_V^r(x)\in\gker(M_{V_{V_m}})$ for any $r>0$, so if $\Delta_l-\Lambda_l\in\gker(M_{V_m})$ then $\sum_{l=0}^{N-1}M_{V_m}^{N-l-1}\left(\Delta_l-\Lambda_l\right)\in\gker(M_{V_m})$.
So for $N$ sufficiently large that $M_{V_\ell}^N[\theta]=0$, this means $M_{V_m}^N[\theta]\in \gker(M_{V_{V_m}})$ and hence $[\theta]\in\gker(M_{V_m})$.

To verify that $\Delta_l-\Lambda_l\in\gker(M_{V_m})$, we first recall that $\Delta_l\in\ker(i_m)$ and $\Lambda_l\in\ker(i_\ell)$.
But $\iota_{\ell m}\circ i_\ell = i_m\circ \iota_{\ell m}$, so $\im(i_\ell)\subseteq \im(i_m)$ and hence $\Delta_l-\Lambda_l\in\ker(i_m)$.
Thus, it remains to show that $\ker(i_m)\subseteq\gker(M_{V_m})$.
For this purpose, recall from Remark~\ref{rem:zig-zag} that $M_{V_m}$, $M_{U_m}$ (and similarly $\tilde{M}_{V_m}$ and $\tilde{M}_{V_m}$) can be constructed from the interval decomposition of zig-zag persistence.
Under this construction, we identify $\ker(i_m)$ with the span of $I[0,b](0)$ for all barcodes of the form $I[0,b]$ for $0\leq b\leq \infty$.
However, this zig-zag persistence module is the lift of a circular-valued zig-zag module, and therefore any infinite-length barcode must be doubly infinite (c.f. \cite[Equation~3]{burghelea2017topology} or \cite[Equation~1]{fersztand2024harder}).
Therefore, we may identify $\ker(i_m)$ with the span of $I[0,b](0)$ for all \emph{finite} barcodes of the form $I[0,b]$.
There must be finitely many such barcodes, as otherwise $\Hg_\bullet(V)\supseteq\bigoplus I[0,b]$ is infinite dimensional, which contradicts the assumption that $K_m$ is locally compact and paracompact.
Then setting $B<\infty$ to be the maximum value of $b$ for all barcodes of the form $I[0,b]$, we must have $\ker(i_m)\subseteq\ker(M_{V_m}^B)\subseteq\gker(M_{V_m})$.

\end{proof}

\section{Discussion}
We have introduced a new method for calculating the homology of $1$-periodic cellular complexes from local, finite data. In particular, one can study homology classes locally through two endomorphisms, $M_U$ and $M_V$, 
which encode translational symmetries and information about the zig-zag persistence of the complex.
This allows us to entirely classify toroidal and non-toroidal cycles of quotient spaces and recover the entire homology of a periodic complex from this information.

To the authors' knowledge, this constitutes the first computationally viable method of calculating the homology of periodic cellular complexes of dimensions greater than 1, though a different approach has been announced in \cite{Heiss-AATRN}. While a full persistence theory remains out of reach, the endomorphisms satisfy a form of unimodality on filtered periodic cellular complexes which provides information about the persistence of toroidal cycles, representing progress in understanding the persistent homology of such spaces.

However, the need to choose bases in defining these endomorphisms provides an obvious limitation for this approach.
As shown in Example~\ref{ex:persistence}, $M_V$ does not commute with inclusion maps of a filtered complex, so we are, in general,  unable to completely and efficiently represent and compute persistent homology in this setting with this approach.
Similarly, this approach does not extend to $d$-periodic complexes for $d>1$ to even classify toroidal and non-toroidal cycles in homology (e.g. without introducing persistence).
Nonetheless, interesting open questions and directions arise from these results.
\begin{enumerate}
\item Though a complete characterization of persistent homology in this setting is unlikely, how can unimodality be exploited for understanding the behavior or these systems?
\item How much information is retained/lost if a $d$-periodic complex is studied by its component 1-periodic complexes?
\item While endomorphisms have been widely studied, in applied topology as well as other areas, is there an overarching framework to understand this setting from the perspective of persistent homology?
\end{enumerate}

\section*{Acknowledgements}\label{ackref}
A.O.~acknowledges the support of the Additional Funding Programme for Mathematical Sciences, delivered by EPSRC (EP/V521917/1) and the Heilbronn Institute for Mathematical Research. The authors would also like to thank Vanessa Robins for many helpful discussions in the overlying goal of computing the persistent homology of periodic point sets.

\bibliographystyle{apalike}
\bibliography{refs}

\appendix

\section{The Mayer-Vietoris Spectral Sequence} \label{sec:MVSS}
In this Appendix we define and summarise machinery to do with the Mayer-Vietoris Spectral Sequence used in the proof of Theorem~\ref{thm:isomorphisms}.
For a more formal introduction we direct the reader to \cite{onus2022quantifying} and for further details we direct the reader to \cite{brown,lewis2014multicore,stafa2015mayer}.

Let $K$ be a cellular complex and let $\U=\{U_{i\in I}\}$ be a cover of $K$ by subcomplexes such that only finite intersections of sets in $\U$ may be non-empty.
Let $\nerve$ denote the simplicial complex defined by the nerve of the cover $\U$. The \textit{blow-up} complex of $K$ with respect to $\U$ is the $\Z$-bigraded module $E^0=\{E^0_{p,q}\}_{p,q\in\Z}$, where $E^0_{p,q}$ is a subspace of the tensor product $C_p(\nerve)\otimes C_q(K)$ generated by elements of the form $J\otimes \gamma$ with $\gamma\in C_q\left(\cap_{j\in J}U_j\right)$
\[
E^0_{p,q}:= \langle J\otimes \gamma\,:\,\gamma\in C_q\left(\cap_{j\in J}U_j\right),\,|J|=p+1\rangle .
\]
Here, we slightly abuse notation by writing $J$ in $J\otimes \gamma$ to denote the $p$-simplex in the nerve of $\U$ representing the non-empty intersection $\cap_{j\in J}U_j$.
$E^0$ is equipped with the maps $\D^0_{p,q}:E^0_{p,q}\to E^0_{p,q-1}$ and $\D^1_{p,q}:E^0_{p,q}\to E^0_{p-1,q}$ induced by the boundary maps on $K$ and nerve of $\U$ respectively.
The maps $\D^0$ and $\D^1$ satisfy $(\D^0)^2= 0, (\D^1)^2=0, \D^0\D^1+\D^1\D^0=0$, which makes the blow-up complex $(E^0,\D^0,\D^1)$ a \textit{bi-complex}.
Bi-complexes are special because their diagonals define a graded complex, called the \textit{total complex}.
Given $K$ and $\U$, the associated total complex is denoted $(T_{\U,K},\D)$, where $(T_{\U,K})_k:=\bigoplus_{p+q=k}E^0_{p,q}$ and $\D^2=(\D^0+\D^1)^2=0$.

A \textit{spectral sequence} $(E^r,d^r)$ is a collection of $\Z$-bigraded modules $E^r$ and differential maps $d^r$ with the property that $H(E^r,d^r)=E^{r+1}$ (that is, the homology with respect to $d^r$ of $E^r$ determines $E^{r+1}$).
We say that $E^r$ is the $r^\mathrm{th}$-\textit{page} of the spectral sequence.
If for each $p,q\in\Z$ there exists an $r_{p,q}$ such that $E^r_{p,q}\cong E^{r_{p,q}}_{p,q}$ for $r \geq r_{p,q}$, then we define $E^\infty=\{E^{r_{p,q}}_{p,q}\}_{p,q\in\Z}$ to be the $\infty$-\textit{page} of the spectral sequence and say that $(E^r,d^r)$ converges to $E^\infty$.

Every bi-complex induces a spectral sequence.
In particular, for $K$ and $\U$, the total complex describes the 0th page of a spectral sequence called the \textit{Mayer-Vietoris spectral sequence} (MVSS).
The Mayer-Vietoris spectral sequence is so named because it generalises the standard exact sequence to covers with more than two elements. 
We can determine the homology of $K$ from the diagonals of $E^\infty$ or from $E^0$ taken as the total complex (i.e. the diagonals of $E^0$) by a result of \cite{godement1958topologie}.
\begin{theorem} \label{thm:mvss}
For homology over a field $\F$,
\[
\Hg_k(K)\cong \Hg_k(T_{\U,K},\D)\cong\bigoplus_{p+q=k}E^\infty_{p,q}.
\]
\end{theorem}
\begin{remark}
Theorem~\ref{thm:mvss} does not require any finiteness assumptions on $K$.
To see this for the first isomorphism, we direct the reader to \cite[Section~2]{dugger2004topological}, noting that we may equivalently consider $\U$ an open cover of $K$ since every $U\in\U$ has an open neighbourhood onto which it deformation retracts by \cite[Proposition~A.5]{hatcher}.
For the second isomorphism, we direct the reader to \cite[Section~2.2]{mccleary2001user} and in particular Theorem~2.6, where the isomorphism follows by a column filtration of the total complex, noting that all short exact sequences of (non-topological, but possibly infinite dimensional) vector fields are split exact by an application of Zorn's Lemma.
\end{remark}
The isomorphisms in Theorem~\ref{thm:mvss} are explicitly defined 
as follows\footnote{It is a standard graduate level course exercise to show that both maps are well-defined isomorphisms, but for completeness we still make the maps explicit.}.
To begin, suppose $(\alpha_0,\dots,\alpha_k)$ represents an element of $\Hg_k(T_{\U,K},\D)$, where $\alpha_p\in E^0_{p,k-p}$ for $p=0,\dots,k$.
If $\U=\{U_i, i\in\mathcal{I}\}$ then, then $\alpha_0$ decomposes to the sum $\sum_{i\in\mathcal{I}}(\{i\}\otimes\gamma_i)$.
Writing $\alpha = \sum_{i\in \mathcal{I}}\gamma_i$ in $K$,
Then the isomorphism $\Hg_k(K)\cong \Hg_k(T_{\U,K},\D)$ maps $[(\alpha_0,\dots,\alpha_k)]\in\Hg_k(T_{\U,K},\D)$ to $[\alpha]\in \Hg_k(K)$.
This is summarised by the first equivalence in Equation~\ref{eq:correspondence}
Then the second isomorphism $\Hg_k(T_{\U,K},\D)\cong\bigoplus_{p+q=k}E^\infty_{p,q}$ comes from a standard column filtration of the total complex.
Explicitly, $E^\infty_{p,k-p}$ represents elements of $\Hg_k(T_{\U,K},\D^0+\D^1)$ with representatives in $\bigoplus_{i=0}^pE^0_{i,k-i}\subseteq T_{\U,K}$ which do not have representatives in $\bigoplus_{i=0}^{p-1}E^0_{i,k-i}$.
To this end, for each $p=0,\dots,k$ and each non-trivial $[\alpha_p]\in E^\infty_{p,k-p}$, there exists $[(\alpha_0,\dots,\alpha_p,0,\dots,0)]\in \Hg_k(T_\U,\D^0+\D^1)$
This is summarised by the second (reverse) implication in Equation~\ref{eq:correspondence}.
\begin{equation} \label{eq:correspondence}
[\alpha] \in \Hg_k(K) \;\Longleftrightarrow\; \exists\, [(\alpha_0,\dots,\alpha_p,0,\dots,0)]\in \Hg_k(T_{\U,K},\D^0+\D^1) \;\Longleftarrow\;[\alpha_p]\in E^\infty_{p,k-p}
\end{equation}

\end{document}